\documentclass[fleqn,reqno]{amsart}    
\usepackage{amsthm}
\usepackage{amssymb}
\usepackage{amsmath}

\usepackage[pdftex]{graphicx}

\usepackage{hyperref}
\hypersetup{colorlinks=true}

\numberwithin{equation}{section}

\newtheorem{theorem}{Theorem}[section]

\newtheorem{lemma}[theorem]{Lemma}
\newtheorem{corollary}[theorem]{Corollary}

\theoremstyle{definition}
\newtheorem{definition}[theorem]{Definition}
\newtheorem{remark}[theorem]{Remark}
\newtheorem{example}[theorem]{Example}

\newcommand{\K}{\mathbb{K}}
\newcommand{\M}{\mathbb{M}}
\newcommand{\N}{\mathbb{N}}

\newcommand{\R}{\mathbb{R}}
\newcommand{\Z}{\mathbb{Z}}

\newcommand{\GL}{\mathbb{GL}}
\newcommand{\Ort}{\mathbb{O}}
\newcommand{\SO}{\mathbb{SO}}

\newcommand{\Sym}{\mathbb{S}}

\newcommand{\cNT}{\mathcal{NT}}
\newcommand{\cR}{\mathcal{R}}

\newcommand{\cT}{\mathcal{T}}

\newcommand{\eqdf}{\mathrel{\mathop:}=}

\newcommand{\abs}[1]{\left\vert#1\right\vert}
\newcommand{\Bif}{\mathrm{Bif}}
\newcommand{\cl}[1]{\mathrm{cl}\left(#1\right)}
\newcommand{\cond}{\;\vert\;\;\;}
\newcommand{\DEG}{\mathrm{DEG}}
\newcommand{\diag}[1]{\mathrm{diag}\left(#1\right)}

\newcommand{\GlBif}{\mathrm{GlBif}}
\newcommand{\halfline}{(0,+\infty)}
\newcommand{\hilb}{H^{1}_{2\pi}}
\newcommand{\I}{\mathrm{I}}
\newcommand{\im}{\mathrm{im}}
\newcommand{\ind}[2]{\mathrm{i}\left(#1,#2\right)}
\newcommand{\IND}{\mathrm{IND}}
\newcommand{\inn}[2]{\left\langle#1,#2\right\rangle}
\newcommand{\Inter}[1]{\mathrm{Int}\left(#1\right)}
\newcommand{\Jac}{\mathrm{Jac}}
\newcommand{\morse}[1]{\mathrm{m}^{-}\left(#1\right)}
\newcommand{\mult}[1]{\mu(#1)}
\newcommand{\norm}[1]{\left\Vert#1\right\Vert}
\newcommand{\nspect}[1]{\sigma_{-}\left(#1\right)}
\newcommand{\pspect}[1]{\sigma_{+}\left(#1\right)}

\newcommand{\Psing}{\mathrm{P_{sing}}}
\newcommand{\set}[1]{\left\{#1\right\}}
\newcommand{\sgn}[1]{\mathrm{sgn}#1}

\newcommand{\spect}[1]{\sigma\left(#1\right)}

\begin{document}

\title[On the structure of the set of bifurcation points]{On the
structure of the set of bifurcation points of periodic solutions
for multiparameter Hamiltonian systems}

\author{Wiktor Radzki}

\thanks{Partially supported by Ministry of Science and Education,
Poland, under grant 1 P03A 009 27, and grant 471-M of Nicolaus
Copernicus University, Toru\'{n}, Poland}

\address{Faculty of Mathematics and Computer Science \\
  Nicolaus Copernicus University \\
  ul. Chopina $12 \slash 18$ \\
  87-100 Toru\'{n}, Poland }

\email{wiktorradzki@yahoo.com}



\date{September 28, 2008}

\begin{abstract}
This paper deals with periodic solutions of the Hamilton
equation
 $\dot{x}(t)=J\nabla_x H(x(t),\lambda),$
where $H\in C^{2,0}(\mathbb{R}^{2n}\times\mathbb{R}^k,\mathbb{R})$
and $\lambda\in\mathbb{R}^k$ is a~parameter. Theorems on global
bifurcation of solutions with periods $\frac{2\pi}{j},$
$j\in\mathbb{N},$ from a~stationary point
 $(x_0,\lambda_0)\in\mathbb{R}^{2n}\times\mathbb{R}^k$
are proved.
$\nabla_x^2 H(x_0,\lambda_0)$ can be singular. However, it is
assumed that the local topological
degree of $\nabla_x H(\cdot,\lambda_0)$ at $x_0$ is nonzero.
For systems satisfying $\nabla_x H(x_0,\lambda)=0$ for all
$\lambda\in\mathbb{R}^k$ it is shown that (global) bifurcation
points of solutions with periods $\frac{2\pi}{j}$ can be identified
with zeros of appropriate continuous
functions $F_j\colon\mathbb{R}^k\rightarrow\mathbb{R}.$
If, for all $\lambda\in\mathbb{R}^k,$
$\nabla_x^2H(x_0,\lambda)=\mathrm{diag}(A(\lambda),B(\lambda)),$
where $A(\lambda)$ and $B(\lambda)$ are $(n\times n)$-matrices,
then $F_j$ can be defined by
$F_j(\lambda)=\det[A(\lambda)B(\lambda)-j^2I].$
Symmetry breaking results concerning bifurcation of solutions
with different minimal periods are obtained.
A~geometric description of the set of bifurcation points is given.
Examples of constructive application of the theorems proved to
analytical and numerical investigation and visualization of the set
of all bifurcation points in given domain are provided.

This paper is based on a~part of the author's thesis
[W. Radzki, \emph{Branching points of periodic solutions of
autonomous Hamiltonian systems} (Polish), PhD thesis,
Nicolaus Copernicus University, Faculty of Mathematics
and Computer Science, Toru\'{n}, 2005].
\end{abstract}

\maketitle

\section{Introduction}

The aim of this paper is to describe the set of bifurcation
points of solutions of the Hamilton equation with the condition of
\mbox{$2\pi$-perio}dicy of solutions
\begin{equation}{\label{parhamniel}}
\left\{
\begin{aligned}
\dot{x}(t) &=J\nabla_x H(x(t),\lambda)\\
x(0) &=x(2\pi),
\end{aligned}
\right.
\end{equation}
where $H\in C^{2,0}(\R^{2n}\times\R^k,\R)$ and $\lambda\in\R^k$ is
a~parameter. In particular, this work is intended to investigate
the subsets of the set of bifurcation points consisting of global
bifurcation points of solutions with periods $\frac{2\pi}{j},$
$j\in\N,$ and to prove theorems concerning symmetry breaking
points, defined as bifurcation points of solutions with different
minimal periods.

In the case of the systems with linear dependence on one parameter
problem~\eqref{parhamniel} can be written as
\begin{equation}{\label{parham}}
\left\{
\begin{aligned}
\dot{x}(t) &=\lambda J\nabla H(x(t))\\
x(0) &=x(2\pi),
\end{aligned}
\right.
\end{equation}
where $H\in C^2(\R^{2n},\R)$ and $\lambda\in\R.$ Every solution
$(x,\lambda)$ of~\eqref{parham} with $\lambda>0$ can be translated
to \mbox{$2\pi\lambda$-pe}riodic solution of the equation
\begin{equation}{\label{ham}}
\dot{x}(t)=J\nabla H(x(t)).
\end{equation}
Consequently, for every connected branch of nontrivial
solutions of~\eqref{parham} bifurcating (in a~suitable space) from
 $(x_0,\lambda_0)\in (\nabla H)^{-1}(\set{0}) \times\halfline$
one can find the corresponding connected branch of
nonstationary periodic trajectories of~\eqref{ham} emanating from
$x_0$ with periods tending to $2\pi\lambda_0$ at $x_0.$ Particulary
interesting systems are those for which the Hessian matrix of $H$ at
$x_0$ has the block-diagonal form:
$\nabla^2H(x_0)= \diag{A,B},$
where $A$ and $B$ are real symmetric \mbox{$(n\times n)$-ma}trices.
This condition is satisfied in the generic case of
Hamiltonian function being the sum of kinetic energy dependent on
generalized momenta and potential energy dependent on generalized
coordinates, for example if
\begin{equation}{\label{form}}
H(x)=H(y,z)= \frac{1}{2}\langle M^{-1}y,y\rangle +V(z),
\end{equation}
where $y,z\in\R^n,$ $V\in C^2(\R^n,\R)$ and $M$ is a~nonsingular
real symmetric \mbox{$(n\times n)$-ma}\-trix.
Equation~\eqref{ham} with $H$ given by~\eqref{form} is
equivalent to the Newton equation
\begin{equation}{\label{newton}}
M\ddot{z}(t)=-\nabla V(z(t)).
\end{equation}

If $x_0$ is a stationary point of~\eqref{ham},
$J\nabla ^2H(x_0)$ is nonsingular, and it has nonresonant
purely imaginary eigenvalues then
the Lyapunov centre theorem~\cite{L} ensures the existence of
a~one-parameter family of nonstationary periodic solutions
of~\eqref{ham} emanating from $x_0.$ The Lyapunov centre theorem
can be derived from the Hopf bifurcation theorem~\cite{Hopf}.
Berger~\cite{Br} (see also~\cite{Brg,MW}), Weinstein~\cite{W},
Moser~\cite{M}, and Fadell and Rabinowitz~\cite{FR} proved the
existence of a~sequence of periodic solutions of~\eqref{ham}
convergent to a~nondegenerate stationary point $x_0$ in the case of
possibly resonant purely imaginary eigenvalues
of $J\nabla ^2H(x_0).$ (The theorem of Berger concerns second order
equations, including~\eqref{newton} for $M=I.$) Global bifurcation
theorems in nondegenerate case have been proved by G\c{e}ba and
Marzantowicz~\cite{GM} by using topological degree for
\mbox{$\SO(2)$-equi}variant mappings.

Zhu~\cite{Z} and Szulkin~\cite{S} used Morse theoretic methods and
they proved the existence of a sequence of periodic solutions
of~\eqref{ham} emanating from a~stationary point which can be
degenerate. Dancer and Rybicki~\cite{DR} obtained a~global
bifurcation theorem of Rabinowitz type (see~\cite{Rab})
for~\eqref{parham} in the case of possibly
degenerate stationary point by using the topological degree theory
for \mbox{$\SO(2)$-equi}variant gradient maps. The results
from~\cite{DR} were applied by the author~\cite{Rd} and the author
with Rybicki~\cite{RR} to the description of connected branches of
bifurcation of~\eqref{parham} and emanation of~\eqref{ham}
in possibly degenerate case under assumptions written in terms
of eigenvalues of $\nabla^2H(x_0)$ and the local topological
degree of $\nabla H$ in a~neighbourhood of $x_0.$ The examples
of applications of the results from~\cite{Rd,RR} were given by
Maciejewski, the author, and Rybicki~\cite{RRM}.

The structure of the set of bifurcation points of periodic solutions
of the first order ordinary differential equations with many
parameters was studied by Izydorek and Rybicki~\cite{IR}, and
Rybicki~\cite{RyDIE}. They applied the Krasnosiel'skii
bifurcation theorem (see~\cite{Kra}) and the results of real
algebraic geometry obtained by Szafraniec~\cite{Sf,SfMA}.
However, Izydorek and Rybicki assumed that the Fr\'{e}chet
derivative of the right-hand side of the equation they considered
was zero. In such a~case there is no bifurcation of nonstationary
solutions of Hamiltonian system with fixed period
(see Remark~\ref{remnec}).

In the present paper, which presents the results of a~part of the
author's PhD thesis~\cite{RPhD} (with Corollaries~\ref{infpercor},
\ref{multzdeg}, \ref{multstat}, Examples~\ref{surfdeg},
\ref{surfstat}, and figures added afterwards),
the stationary point $(x_0,\lambda_0)$ can be degenerate,
i.e. $\nabla_x^2H(x_0,\lambda_0)$ can be singular.
However, it is assumed that the local Brouwer degree of
$\nabla_x H(\cdot,\lambda_0)$ in a~neighbourhood of $x_0$
is well defined and nonzero.
(Although theorems without this assumptions and corresponding
examples are also given.) The set of bifurcation points
of~\eqref{parhamniel} is investigated in the case of many
parameters. To this aim a~generalized version
of the global bifurcation
theorem of Dancer and Rybicki~\cite{DR} in the case of Hamiltonian
systems with one parameter is first proved and then it is applied
to the Hamiltonian systems with many parameters.
Also, some results from~\cite{Rd,RR} concerning unbounded
branches of periodic solutions are generalized.
The proofs
exploit the topological degree for \mbox{$\SO(2)$-equi}variant
gradient mappings (see~\cite{Ry}).
Bifurcation points of solutions of~\eqref{parhamniel} with
period $\frac{2\pi}{j},$ $j\in\N,$ (proved to be global bifurcation
points) are identified with zeros of suitable continuous functions
$F_j\colon\R^k\to\R,$ under assumptions written in terms of that
functions. In the case of systems satisfying, for all
$\lambda\in\R^k,$ the condition
 $\nabla_x^2H(x_0,\lambda)=\diag{A(\lambda),B(\lambda)},$
where $A(\lambda)$ and $B(\lambda)$ are some $(n\times n)$-matrices,
the functions $F_j$ are given by
 $F_j(\lambda)=\det[A(\lambda)B(\lambda)-j^2I].$
Symmetry breaking results are obtained.
A~geometric description of the set of bifurcation points
is obtained by using results of real algebraic
geometry~\cite{Sf,SfMA}.
Examples of application of theorems proved in this paper
to analytical and numerical investigation and visualization
of the set of all bifurcation points in given domain are provided.
They demonstrate constructive character of the results obtained in
this paper by using topological degree.

\section{Preliminaries}

In this section notation and terminology are set up and basic
results used in this paper are summarized to make the exposition
self-contained.

\subsection{Algebraic notation}

Let $\M(n,\R)$ be the set of all real \mbox{$(n\times n)$-ma}trices
and let $\GL(n,\R),$ $\Sym(n,\R),$ $\Ort(n,\R)$ be the subsets of
$\M(n,\R)$ consisting of nonsingular, symmetric, and orthogonal
matrices, respectively. For given $n\in\N$ the identity
\mbox{$(n\times n)$-ma}trix is denoted by $I\equiv I_n,$ whereas
\begin{equation*}
J\equiv J_n\eqdf\left[\begin{array}{cc}
 0&-I_n\\
 I_n&0
 \end{array}\right].
\end{equation*}
For any square matrices $A_1,\ldots,A_m$ the symbol
$\diag{A_1,\ldots,A_m}$ stands for the \mbox{block-diag}\-onal matrix
built from $A_1,\ldots,A_m.$

If $A\in\M(n,\R)$ then $\spect{A}$ denotes the spectrum of $A,$
whereas $\pspect{A}$ and $\nspect{A}$ are the sets of real
positive and real negative eigenvalues of $A,$ respectively.
If $\alpha\in\spect{A}$ then $\mult{\alpha}\equiv\mu_A(\alpha)$
denotes the algebraic multiplicity of $\alpha.$ The negative
and the positive Morse index of $A\in\Sym(n,\R)$ are defined as
\begin{equation*}
\morse{A} \eqdf \sum_{\alpha\in\nspect{A}}\mult{\alpha}, \qquad
m^+(A) \eqdf \sum_{\alpha\in\pspect{A}}\mult{\alpha},
\end{equation*}
respectively.

Let a~representation of a~group $G$ on a~linear space $V$ be given.
For every subgroup $H$ of $G$ and every subset $\Omega$ of $V$
it is assumed
\begin{equation*}
\Omega^H  \eqdf \set{v\in \Omega\cond \forall_{h\in H}\;hv=v}
=\set{v\in \Omega\cond H\subset G_v},
\end{equation*}
\begin{equation*}
\Omega_H  \eqdf  \set{v\in \Omega\cond H=G_v},
\end{equation*}
where $G_v$ is the isotropy group of $v.$
Consider another representation of $G$ on a~linear space $W$ and let
$f\colon V\rightarrow W$ be a~\mbox{$G$-equi}\-variant map.
As well known, $G_v\subset G_{f(v)}$ for every $v\in V.$ If $H$ is
a~subgroup of $G$ then $f^H$ denotes the restriction of $f$ to the
pair $(V^H,W^H).$

For $j\in\N\cup\set{0}$ set
\begin{equation*}
\K_j=
\begin{cases}
\Z_j &\text{if $j\in\N$},\\
\SO(2) &\text{if $j=0$}.
\end{cases}
\end{equation*}
For every $j\in\N$ let
$\rho_j\colon\SO(2)\rightarrow \Ort(2,\R)$ be the homomorphism
defined by
\begin{equation*}
\rho_j\left( \left[
\begin{array}{cc}
\cos\phi & -\sin\phi \\
\sin\phi & \cos\phi
\end{array}
\right] \right)= \left[
\begin{array}{cc}
\cos j\phi & -\sin j\phi \\
\sin j\phi & \cos j\phi
\end{array}
\right],\;\;\; 0\leq \phi <2\pi.
\end{equation*}
For $m,j\in\N$ the representation $\R[m,j]$ of $\SO(2)$ is defined
as the direct sum of $m$ copies of the representation
$(\R^2,\rho_j),$ whereas $\R[m,0]$ denotes the identity
representation of $\SO(2)$ on $\R^m.$

Fix $k\in\N.$ For every $j\in\N,$ $K\in\Sym(2n,\R),$ and every
$T\colon\R^k\rightarrow\Sym(2n,\R)$ set
\begin{equation}{\label{oznQ}}
Q_j(K)=\frac{1}{1+j^2}\left[
\begin{array}{cc}
-K&jJ^t\\
jJ&-K
\end{array}
\right], \;\;\;Q_0(K)=-K,
\end{equation}
\begin{equation}{\label{oznLj}}
\begin{aligned}
\Lambda_j(T) &=\{\lambda\in\R^k\cond\det
Q_j(T(\lambda))=0\},\\
\Lambda_0(T) &=\{\lambda\in\R^k\cond\det Q_0(T(\lambda))=0\},
\end{aligned}
\end{equation}
\begin{equation}{\label{oznL}}
\displaystyle\Lambda(T)
=\bigcup_{j\in\N}\Lambda_j(T).
\end{equation}
For $k=1$ let
\begin{equation}{\label{oznLplus}}
\begin{aligned}
\Lambda_j^+(T) &=\Lambda_j(T)\cap\halfline,\\
\Lambda^+(T) &=\Lambda(T)\cap\halfline.
\end{aligned}
\end{equation}

\begin{remark}{\label{qjeven}}
For every $j\in\N$ the eigenvalues of $Q_j(K)$ have even
multiplicity. Indeed, if
 $(v_1,v_2)\in\R^{2n}\times\R^{2n}\backslash\set{(0,0)}$ is
an eigenvector of $Q_j(K)$ then $(v_1-v_2,v_1+v_2)$ is
an eigenvector corresponding to the same eigenvalue.
\end{remark}

\subsection{Degree for \mbox{$\SO(2)$-equi}variant gradient
maps}{\label{degree}}

Proofs of global bifurcation theorems in this paper exploit the
topological degree for \mbox{$\SO(2)$-equi}variant gradient
mappings, which is a~special case of the degree
described in~\cite{Ry}. For earlier results concerning equivariant
degree see~\cite{D,IMV89,DGJM,IMV92,GKW} and references therein.

Consider an orthogonal representation of the group $\SO(2)$ on
a~real inner product space $V$ with $\dim V<\infty.$ Let
$\Omega$ be an~\mbox{$\SO(2)$-in}variant bounded open
subset of $V$ and let  $\nabla f\colon V\rightarrow V$
be a~continuous \mbox{$\SO(2)$-equi}variant gradient
mapping such that
$\nabla f(x)\neq 0$ for every $x\in \partial \Omega.$
Then
\begin{equation*}
\displaystyle\DEG(\nabla f,\Omega)
=\set{\DEG_j(\nabla f,\Omega)}_{j\in\N\cup\set{0}}
\end{equation*}
denotes the \mbox{$\SO(2)$-degree} of $\nabla f$
in $\Omega$~\cite{Ry}.
It is an element of the Euler ring of the group
$\SO(2),$ i.e. the ring
\begin{equation*}
U(\SO(2))=\bigoplus_{j\in\N\cup\set{0}}\K_j
\end{equation*}
with addition $+$ and multiplication $\star$ defined for every
 $\set{a_j}_{j=0}^{\infty},\set{b_j}_{j=0}^{\infty}\in U(\SO(2))$
by
 $\set{a_j}_{j=0}^{\infty}+\set{b_j}_{j=0}^{\infty}
 =\set{a_j+b_j}_{j=0}^{\infty}$
and
 $\set{a_j}_{j=0}^{\infty}\star\set{b_j}_{j=0}^{\infty}
 =\set{c_j}_{j=0}^{\infty},$
where $c_0=a_0b_0,$ $c_j=a_0b_j+a_jb_0,$ $j\in\N.$ Notice that
$\Theta=(0,0,\ldots)$ is the neutral element of addition in this
ring. The degree $\DEG$ has properties analogous to the Brouwer
degree (see~\cite{Ry}). However, if $\DEG_j(\nabla f,\Omega)\neq 0$
for some $j\in\N\cup\set{0}$
then
 $(\nabla f)^{-1}(\set{0})\cap\Omega^{\K_j}\neq \emptyset$
(not only
 $(\nabla f)^{-1}(\set{0})\cap\Omega\neq \emptyset$).

If $y_0\in\R^m$ is an isolated zero of a~continuous mapping
$g\colon\R^m\to\R^m$ then the topological index $\ind{g}{y_0}$
of $y_0$ with respect to $g$ is defined as the Brouwer degree
$\deg(g,B(y_0,r),0)$ of $g$ in a~ball $B(y_0,r)\subset\R^m$ centred
at $y_0$ with radius $r>0$ such that
$g^{-1}(\set{0})\cap\cl{B(y_0,r)}=\set{y_0}.$

Analogously, if $x_0$ is an isolated element of
$(\nabla f)^{-1}(\set{0})$ then its index
\begin{equation*}
\I(\nabla f,x_0)
=\set{\I_j(\nabla f,x_0)}_{j\in\N\cup\set{0}}\in U(\SO(2))
\end{equation*}
with respect to $\nabla f$ is defined by the formula
\begin{equation*}
\I(\nabla f,x_0)=\DEG(\nabla f,B(x_0,r)),
\end{equation*}
where $B(x_0,r)\subset V$ is a~ball such that
 $(\nabla f)^{-1}(\set{0})\cap \cl{B(x_0,r)}=\set{x_0}.$

\begin{lemma}[\cite{DR}]{\label{abstrlem}}
Let $V=\R[m,0]\oplus\R[m_1,j_1]\oplus\cdots\oplus\R[m_r,j_r],$ where
$m,m_i,j_i\in\N,$ $0<j_1<\cdots<j_r.$ Assume that $f\in C^2(V,\R)$
is an \mbox{$\SO(2)$-equi}variant map and $x_0$ is an isolated
element of $(\nabla f)^{-1}(\set{0})$ such that
$\nabla^2f(x_0)=\diag{A_0,A_1,\ldots,A_r}$ for some matrix $A_0$ of
dimension $m$ and for nonsingular matrices $A_i$ of dimensions
$2m_i,$ $i=1,\ldots,r.$ Then
\begin{equation*}
\I_0(\nabla f,x_0)=\ind{\nabla f^{\SO(2)}}{x_0}
\end{equation*}
and for every $j\in\N$ one has
\begin{equation*}
\I_j(\nabla f,x_0)=
\begin{cases}
\displaystyle \ind{\nabla f^{\SO(2)}}{x_0} \cdot\frac{\morse{A_i}}{2}&
\text{if $j=j_i$ for some $i\in\set{1,\ldots,r}$}, \\
0 & \text{otherwise}.
\end{cases}
\end{equation*}
\end{lemma}

Assume that $\nabla_x f\colon V\times\R\rightarrow V$ is
a~continuous \mbox{$\SO(2)$-equi}variant gradient (with respect to
$V$) mapping such that $\nabla_x f(x_0,\lambda)=0$ for some fixed
$x_0\in V$ and all $\lambda\in\R.$ Fix $\lambda_0\in\R$ and assume
that for every $\lambda\in\R,$ $\lambda\neq\lambda_0,$ from
a~neighbourhood of $\lambda_0$ there exists a~neighbourhood
$W\subset V\times\R$ of $(x_0,\lambda )$ such that
 $(\nabla_x f)^{-1}(\set{0})\cap W\subset\set{x_0}\times\R.$
Then for sufficiently small
$\varepsilon>0$ one can define the \emph{bifurcation index}
\begin{equation*}
\IND(x_0,\lambda_0)
=\set{\IND_j(x_0,\lambda_0)}_{j\in\N\cup\set{0}}\in U(\SO(2))
\end{equation*}
\emph{of $(x_0,\lambda_0)$ with respect to $\nabla_x f$} by
\begin{equation}{\label{abstrind}}
\IND(x_0,\lambda_0)
=\I(\nabla_x f(\cdot,\lambda_0+\varepsilon),x_0)
-\I(\nabla_x f(\cdot,\lambda_0-\varepsilon),x_0).
\end{equation}
This bifurcation index will be used in the proof of
Theorem~\ref{genbif}.

\subsection{Functional setting}

For given Hilbert spaces $Y,$ $E,$ $Z$ the symbols
$C^{1,0}(Y\times E,Z)$ and $C^{2,0}(Y\times E,Z)$
denote the sets of continuous functions from $Y\times E$ to $Z$
having, respectively, first partial Fr\'{e}chet derivative and two first partial Fr\'{e}chet derivatives with respect to $Y$
continuous on $Y\times E.$

Solutions $(x,\lambda)$ of~\eqref{parhamniel} are regarded
as elements of the space $\hilb\times\R^k.$
(The description of the Sobolev space
 $\hilb\equiv W^{1,2}([0,2\pi],\R^{2n})$
can be found in~\cite{MW}.) The inner product in $\hilb$ is
defined for every $x,y\in\hilb$ by the formula
\begin{equation*}
\inn{x}{y}_{\hilb} =\inn{x}{y}_{L^2_{2\pi}}
+ \inn{\dot{x}}{\dot{y}}_{L^2_{2\pi}}
=\int_{0}^{2\pi}\inn{x(t)}{y(t)}\mathrm{d}t
+\int_{0}^{2\pi}\inn{\dot{x}(t)}{\dot{y}(t)}\mathrm{d}t,
\end{equation*}
where $\dot{x}$ stands for the weak derivative of $x$ and
$\inn{\cdot}{\cdot}$ denotes the standard inner product
in $\R^{2n}.$ Since every $x\in\hilb$ has a~continuous
representative (denoted by the same symbol) satisfying
the condition $x(0)=x(2\pi),$ it can be regarded as a~continuous
\mbox{$2\pi$-pe}riodic function on $\R.$

For fixed $\lambda\in\R^k$ a~function $x\in\hilb$ is called
a~\emph{weak solution of~\eqref{parhamniel}} if the equation
 $\dot{x}(t)=J\nabla_x H(x(t),\lambda)$
(where $\dot{x}$ denotes the weak derivative of
$x$) is satisfied for almost all $t\in[0,2\pi].$ However, since it
is assumed that $H\in C^{2,0}(\R^{2n}\times\R^k,\R),$ every such
solution is in fact a~classical solution of class $C^2$ on
$[0,2\pi]$ and it has a~unique extension to the classical solution
on $\R,$ which is a~\mbox{$2\pi$-pe}riodic function of class $C^2.$

Let $Y,$ $Z$ be Hilbert spaces. Consider a~map
$F\colon Y\times\R^k\rightarrow Z$ and a~fixed set $\Delta\subset Y$
such that $F(x,\lambda)=0$ for all $x\in \Delta,$
$\lambda\in\R^k.$ The set  $\Delta\times\R^k$ is referred to as
the set of \emph{trivial solutions} of the equation
\begin{equation}{\label{genbifpr}}
F(x,\lambda)=0.
\end{equation}
The complement of $\Delta\times\R^k$ in the set of all solutions
of~\eqref{genbifpr} in $Y\times\R^k$ is called the set of
\emph{nontrivial solutions}.

\begin{definition}{\label{genbifdef}}
Let $X\subset Y\times\R^k$ be a~subset of the set of nontrivial
solutions of~\eqref{genbifpr}. A~solution
$(x_0,\lambda_0)\in\Delta\times\R^k$ is called a~\emph{bifurcation
point of solutions from $X$} if it is a~cluster point of $X.$ It
is called a~\emph{branching point of solutions from $X$} if there
exists a~connected set $C\subset X$ such that
$(x_0,\lambda_0)\in\cl{C}$ (the closure in $Y\times\R^k$).
If the connected component of $\cl{X}$ containing the
bifurcation point $(x_0,\lambda_0)$ is unbounded or it contains
another bifurcation point of solutions from $X$ then
$(x_0,\lambda_0)$ is said to be a~\emph{global bifurcation point
of solutions from $X$}.
\end{definition}

Assuming $Y=\hilb,$ $Z=L^2_{2\pi},$ and defining
$F\colon \hilb\times\R^k\rightarrow L^2_{2\pi}$ by
\begin{equation*}
F(x,\lambda)(t)=\dot{x}(t)-J\nabla_xH(x(t),\lambda)
\end{equation*}
one can write~\eqref{parhamniel} in form~\eqref{genbifpr}, therefore
Definition~\ref{genbifdef} can be applied. If $(x,\lambda)$ is
a~\emph{stationary solution} of~\eqref{parhamniel}, i.e. $x$ is
constant, then $x$ is regarded as an element of $\R^{2n}.$ For fixed
$x_0\in\R^{2n}$ such that
\begin{equation}{\label{stationary}}
\nabla_xH(x_0,\lambda)=0 \;\;\;\text{for all $\lambda\in\R^k$}
\end{equation}
one can assume $\Delta=\set{x_0}.$ In such a~case the set
$\set{x_0}\times\R^k$ of trivial solutions of~\eqref{parhamniel}
is denoted by $\cT(x_0)$ and the symbol $\cNT(x_0)$ stands for
the set of nontrivial solutions of~\eqref{parhamniel}. Notice that
$\cNT(x_0)$ can contain stationary solutions.

Define the action of $\SO(2)$ on $\hilb$ as follows. For every
$x\in\hilb$ and
\begin{equation}{\label{gform}}
g= \left[
\begin{array}{cc}
\cos\phi&-\sin\phi\\
\sin\phi&\cos\phi
\end{array}
\right]\in\SO(2),\;\;\;0\leq\phi<2\pi,
\end{equation}
set
\begin{equation*}
(gx)(t)=x(t+\phi).
\end{equation*}
The space $\hilb\times\R^k$ is regarded as the direct sum of the
orthogonal representation of $\SO(2)$ on $\hilb$ defined above and
the identity representation of $\SO(2)$ on $\R^k.$ One has
$\SO(2)_{(x,\lambda)}=\SO(2)_x$ for every
$(x,\lambda)\in\hilb\times\R^k.$

The subspaces of $\hilb$ defined as
\begin{align*}
E_0 & \eqdf
\set{x\in\hilb\cond x(t)\equiv a, a\in\R^{2n}},\\
E_j & \eqdf  \set{x\in\hilb\cond x(t) \equiv a\cos jt
+b\sin jt,\; a,b\in\R^{2n}},\;\;\;j\in\N,
\end{align*}
are \mbox{$\SO(2)$-equi}variant. One has $E_0\approx\R[2n,0]$ and
$E_j\approx\R[2n,j]$ for $j\in\N.$ Obviously,
$(\hilb)^{\SO(2)}=(\hilb)_{\SO(2)}=E_0$ and if $j\in\N,$
$v\in E_j\backslash\set{0},$ then $\SO(2)_v=\Z_j.$ Furthermore,
\begin{equation*}
(\hilb)^{\Z_j}=\bigoplus_{l\in\N\cup\set{0}}E_{lj}
\approx \bigoplus_{l\in\N\cup\set{0}}\R[2n,lj].
\end{equation*}

Let $(e_1,\ldots,e_{2n})$ be the standard basis in $\R^{2n}.$
For fixed $j\in\N$ set
$\varphi_0(t)\equiv 1,$ $\varphi_j(t)\equiv\cos jt,$ $\psi_j(t)\equiv\sin jt,$
and
\begin{equation*}
\hat{e}_i=
\begin{cases}
e_i\varphi_j &;1\leq i\leq 2n,\\
e_{i-2n}\psi_j&;2n+1\leq i\leq 4n.
\end{cases}
\end{equation*}
Then $(e_1\varphi_0,\ldots,e_{2n}\varphi_0)$ and
$(\hat{e}_1,\ldots,\hat{e}_{4n})$ are called the \emph{standard
bases in $E_0$} and \emph{$E_j$}, respectively. The standard basis
in $E_{j_1}\oplus\cdots\oplus E_{j_s},$ where
$j_1,\ldots,j_s\in\N\cup\set{0},$ $s\in\N,$ is
built from the standard bases in $E_{j_1},\ldots,E_{j_s}.$

\begin{remark}{\label{mod}}
There exists $c>0$ such that for every $x\in\hilb$ (identified with
its continuous representative) one has
$\norm{x}_0\leq c\norm{x}_{\hilb},$
where
 $\displaystyle \norm{x}_0 =\sup_{t\in[0,2\pi]}\abs{x(t)}$
(see~\cite{MW}). As it was observed in~\cite{DR}, for given
$H\in C^{2,0}(\R^{2n}\times\R^k,\R)$ and
bounded $U\subset\hilb\times\R^k$ one can find $\eta>0$ and
$H_1\in C^{2,0}(\R^{2n}\times\R^k,\R)$ such that
\begin{enumerate}
\item for every $(x,\lambda)\in\cl{U},$
      $t\in[0,2\pi]$ one has
      $(x(t),\lambda)\in
      B(0,\eta)\subset\R^{2n}\times\R^k,$
\item $\displaystyle H_1|_{B(0,\eta)}=
      H|_{B(0,\eta)},\;\;\;
      H_1|_{\R^{2n}\times\R^k\backslash B(0,2\eta)}=0,$
\item $(x,\lambda)\in\cl{U}$ is a~solution of~\eqref{parhamniel}
      iff it is a~solution of the problem
      \begin{equation}{\label{parhammod}}
      \left\{
      \begin{aligned}
      \dot{x}(t) &=J\nabla_x H_1(x(t),\lambda)\\
      x(0) &=x(2\pi).
      \end{aligned}
      \right.
      \end{equation}
\end{enumerate}
Consequently, investigating bounded (in $\hilb\times\R^k$) subsets
of solutions of~\eqref{parhamniel} one can replace $H$ by
a~modified Hamiltonian $H_1$ having compact support, therefore
no growth conditions are needed.
\end{remark}

Theorem~\ref{finitedim} given below has been extracted from
the proof of Theorem~3.3 in~\cite{DR}. It is a~version
of the Amann-Zehnder global reduction~\cite{AZ1,AZ2}.
Every point $x_0\in\R^{2n}$ is identified
with the constant function from $E_0\subset\hilb.$
The gradients $\nabla_xa(\cdot,\lambda),$
$\nabla_xH_1(\cdot,\lambda)$ and the Hessian matrices
$\nabla_x^2 a(x_0,\lambda),$ $\nabla_x^2H_1(x_0,\lambda)$
are computed with respect to the inner product
$\langle\cdot,\cdot\rangle_{\hilb}$
and the standard inner product in $\R^{2n},$ respectively.
Use is made of the standard bases in $E_f$ and $\R^{2n}.$

\begin{theorem}{\label{finitedim}}
If $H_1\in C^{2,0}(\R^{2n}\times\R^k,\R)$ has compact support
then there exist $r_0\in\N$ and an \mbox{$\SO(2)$-equi}variant
mapping $a\in C^{2,0}(E_f\times\R^k,\R),$ where
\begin{equation*}
E_f \eqdf \bigoplus_{j=0}^{r_0}E_j
\approx\bigoplus_{j=0}^{r_0}\R[2n,j],
\end{equation*}
such that for every $x_0\in\R^{2n},$ $\lambda\in\R^k$ the
following conditions are satisfied.
\begin{enumerate}

\item $\displaystyle a(x_0,\lambda)=-2\pi H_1(x_0,\lambda).$

\item $\nabla_x a(x_0,\lambda)=0$ iff
      $\nabla_xH_1(x_0,\lambda)=0.$
      Moreover,
      $\nabla_x a^{\SO(2)} =\nabla_x a|_{(E_0\times\R^k,E_0)}
      =-\nabla_x H_1.$

\item If $\nabla_x a(x_0,\lambda)=0$ then
      \begin{equation*}
      \nabla_x^2 a(x_0,\lambda)=
      \diag{Q_0(\nabla_x^2H_1(x_0,\lambda)),
      \ldots,Q_{r_0}(\nabla_x^2H_1(x_0,\lambda))}.
      \end{equation*}

\item{\label{amzlemmor}} For every $j>r_0$ one has
      $\morse{Q_j(\nabla_x^2H_1(x_0,\lambda))}
      =m^+\left(Q_j(\nabla_x^2H_1(x_0,\lambda))\right)=2n$
      (in particular,
      $\det Q_j(\nabla_x^2H_1(x_0,\lambda))\neq 0$).

\end{enumerate}
Furthermore, there exists an \mbox{$\SO(2)$-equi}variant
homeomorphism
 $h\colon (\nabla_x a)^{-1}(\set{0})\rightarrow \cR(H_1),$
where $\cR(H_1)\subset\hilb\times\R^k$ is the set of solutions
of~\eqref{parhammod}.
\end{theorem}

Conclusion~\eqref{amzlemmor} in the above theorem holds true for
every $x_0\in\R^{2n}$ and $\lambda\in\R^k,$ since it is assumed
that $H_1$ has compact support. The fact that $h$ is
a~homeomorphism follows from its construction. Notice that the
authors of~\cite{AZ1,AZ2} regard the space $\hilb$ as a~subspace
of $L^2_{2\pi}\equiv L^2([0,2\pi],\R^{2n})$ and they use the inner
product
$\langle\cdot,\cdot\rangle_{L^2_{2\pi}}$
which generates weaker topology in $\hilb$ than the inner product
$\langle\cdot,\cdot\rangle_{\hilb}.$ It affects also the form of
matrices $Q_j$ and changes their eigenvalues used in the reduction.
The matrices used in~\cite{DR} are in fact those from~\cite{AZ1,AZ2}
(without the factor $\frac{1}{1+j^2}$). However, the change of the
inner product is possible in view of the following lemma.

\begin{lemma}{\label{homsol}}
Assume that $H_1\in C^{2,0}(\R^{2n}\times\R^k,\R)$ has compact
support, $\cR(H_1)$ is the set of solutions of~\eqref{parhammod},
and $d_{L^2_{2\pi}},$ $d_{\hilb}$ are the metrics in $\cR(H_1)$
induced by the product norms from $L^2_{2\pi}\times\R^k$ and
$\hilb\times\R^k,$ respectively. Then the identity mapping from
$(\cR(H_1),d_{L^2_{2\pi}})$ to $(\cR(H_1),d_{\hilb})$ is
a~homeomorphism.
\end{lemma}
\begin{proof}
It suffices to prove that the identity mapping from the space
$(\cR(H_1),d_{L^2_{2\pi}})$ to the space $(\cR(H_1),d_{\hilb})$ is
continuous. Suppose that a~sequence
$\set{(x_m,\lambda_m)}_{m\in\N}\subset\cR(H_1)$ is convergent to
some $(x,\lambda)\in\cR(H_1)$ with respect to the metric
$d_{L^2_{2\pi}}.$ It will be shown to be also convergent with
respect to the metric $d_{\hilb}.$ Since
\begin{equation*}
\norm{x_m-x}_{\hilb}^2=
\norm{x_m-x}_{L^2_{2\pi}}+\norm{\dot{x}_{m}-\dot{x}}_{L^2_{2\pi}},
\end{equation*}
it remains to prove that
$\norm{\dot{x}_m-\dot{x}}_{L^2_{2\pi}}\to 0$
as $m\rightarrow\infty.$ The mapping $J\nabla_xH_1$ is continuous
and has compact support, hence there exist $a,b>0$ such that for
all $(y,\alpha)\in\R^{2n}\times\R^k$ the growth condition
\begin{equation*}
\abs{J\nabla_xH_1(y,\alpha)} \leq a + b\abs{y}\equiv a
+b\abs{y}^{\frac{2}{2}}
\end{equation*}
is satisfied. Consequently, by a~Krasnosiel'skii theorem, the
mapping
\begin{equation*}
L^2_{2\pi}\times\R^k\ni (z,\alpha) \mapsto
J\nabla_xH_1(z(\cdot),\alpha)\in L^2_{2\pi}
\end{equation*}
is continuous, hence
\begin{equation*}
\norm{\dot{x}_{m}-\dot{x}}_{L^2_{2\pi}}
=\norm{J\nabla_xH_1(x_m(\cdot),\lambda_m)
-J\nabla_xH_1(x(\cdot),\lambda)}_{L^2_{2\pi}}\rightarrow 0
\end{equation*}
as $m\rightarrow\infty.$
\end{proof}

\begin{lemma}{\label{compbranch}}
If $H_1\in C^{2,0}(\R^{2n}\times\R^k,\R)$ has compact support then
the set $\cR(H_1)$ of solutions of~\eqref{parhammod} is closed in
$\hilb\times\R^k$ and every bounded subset of $\cR(H_1)$ is
relatively compact.
\end{lemma}
\begin{proof}
The set $\cR(H_1)$ is closed in $\hilb\times\R^k$ as the set of
critical points of the action functional of class $C^{2,0}$ defined
on $\hilb\times\R^k$ (see~\cite{DR}). If the action functional is
defined on $H^{\frac{1}{2}}_{2\pi}\times\R^k$ (see~\cite{RA,Abb})
then it is still of class $C^{2,0},$ the set of its critical points
is still equal to $\cR(H_1),$ and its gradient is a~compact
perturbation of a~selfadjoint Fredholm operator. (In the case of
the space $\hilb\times\R^k$ the Hessian operator of the action
functional is compact not Fredholm.) Thus $\cR(H_1)$ is closed in
$H^{\frac{1}{2}}_{2\pi}\times\R^k$ and bounded subsets of $\cR(H_1)$
are relatively compact in $H^{\frac{1}{2}}_{2\pi}\times\R^k$.
However, those subsets of $\cR(H_1)$ that are bounded in
$\hilb\times\R^k$ are also bounded in
$H^{\frac{1}{2}}_{2\pi}\times\R^k$ and both topologies restricted to
$\cR(H_1)$ coincide, in view of Lemma~\ref{homsol}.
\end{proof}

\section{Necessary conditions for bifurcation and symmetry breaking}

\begin{remark}{\label{isper}}
Let $x\in\hilb$ and $j\in\N.$ Then
\begin{enumerate}
\item{\label{assstat}} $\SO(2)_x=\SO(2)$ iff $x$ is a~constant
      function,
\item{\label{asscont}} $\SO(2)_x\supset\Z_j$ iff
      $\frac{2\pi}{j}$ is a~period (not necessarily minimal)
      of $x,$
\item{\label{assequ}} $\SO(2)_x=\Z_j$ iff $\frac{2\pi}{j}$ is
      the minimal period of $x.$
\end{enumerate}
Equivalence~\eqref{assstat} is straightforward.

To see~\eqref{asscont} first observe that $g\in\Z_j$ iff $g$ is of
form~\eqref{gform} with $\phi=\frac{2\pi}{j}k$ for some
$k\in\set{0,\ldots,j-1}.$ For such a~$g$ one has
\begin{equation}{\label{ok}}
(gx)(t)=x\left(t+\frac{2\pi}{j}k\right).
\end{equation}
If $\SO(2)_x\supset\Z_j$ then~\eqref{ok} implies
\begin{equation}{\label{okey}}
x\left(t+\frac{2\pi}{j}k\right)=x(t).
\end{equation}
In particular, putting $k=1$ one finds that $\frac{2\pi}{j}$ is
a~period of $x.$ Conversely, if $\frac{2\pi}{j}$ is a~period of $x,$
then~\eqref{okey} is satisfied for every $k\in\Z,$ hence~\eqref{ok}
implies $\SO(2)_x\supset\Z_j.$

Now, turn to assertion~\eqref{assequ}. If $\SO(2)_x=\Z_j$ then
$\frac{2\pi}{j}$ is a~period of $x,$ in view of~\eqref{asscont}.
If there was a~smaller period of $x$ then it would
be equal to $\frac{2\pi}{m}$ for some $m\in\N,$ $m>j,$ since $x$
is \mbox{$2\pi$-pe}riodic. Then~\eqref{asscont} would imply
$\Z_m\subset\SO(2)_x=\Z_j,$ a~contradiction. Conversely, if
$\frac{2\pi}{j}$ is the minimal period of $x$ then
$\SO(2)_x\supset\Z_j,$ according to~\eqref{asscont}. Moreover,
for every $g\in\SO(2)_x$ of form~\eqref{gform} $\phi$ is a~period
of $x.$ Thus $\phi$ has to be an integer multiple of
$\frac{2\pi}{j},$ hence $g\in\Z_j.$ Consequently,
$\SO(2)_x\subset\Z_j.$
\end{remark}

\begin{definition}
Let $j\in\N\cup\set{0}.$ A~solution $(x,\lambda)$
of~\eqref{parhamniel} is called a~\emph{\mbox{$j$-so}lution} if
$\SO(2)_{(x,\lambda)}\equiv\SO(2)_x\supset\K_j.$
\end{definition}

If $j\in\N$ then $(x,\lambda)$ is a~\mbox{$j$-so}lution
of~\eqref{parhamniel} iff $\frac{2\pi}{j}$ is a~period (not
necessarily minimal) of $x,$ whereas \mbox{$0$-so}lutions are the
stationary ones.

In the reminder of this section $x_0$ satisfying~\eqref{stationary}
is fixed and $\cT(x_0)=\set{x_0}\times\R^k$ is regarded as the set
of trivial solutions of~\eqref{parhamniel}.

\begin{definition}{\label{symbrdef}}
A~point $(x_0,\lambda_0)\in\cT(x_0)$ is called a~\emph{symmetry
breaking point for~\eqref{parhamniel}} if every neighbourhood of
$(x_0,\lambda_0)$ in $\hilb\times\R^k$ contains at least two
nontrivial solutions of~\eqref{parhamniel} with different isotropy
groups (or, equivalently, different minimal periods -- see
Remark~\ref{isper}).
\end{definition}

Proofs of theorems on symmetry breaking in this paper exploit the
following lemma, based on a~remark from~\cite{Dis}.

\begin{lemma}{\label{isotropy}}
If $H\in C^{2,0}(\R^{2n}\times\R^k,\R)$ then for every
$\lambda_0\in\R^k$ there exists a~neighbourhood
$U\subset\hilb\times\R^k$ of $(x_0,\lambda_0)\in\cT(x_0)$ such that
the isotropy group $\SO(2)_{(x,\lambda)}=\SO(2)_x$ of every
nontrivial solution $(x,\lambda)\in U\cap \cNT(x_0)$
of~\eqref{parhamniel} belongs to the set $G(\lambda_0)$ of isotropy
groups of nonzero elements of the finite dimensional space
$E(\lambda_0)=\displaystyle \bigoplus_{j\in X(\lambda_0)}E_j,$ where
\begin{equation*}
X(\lambda_0) =\set{j\in\N\cup\set{0}\cond
\det Q_j(\nabla_x^2H(x_0,\lambda_0))=0}.
\end{equation*}
\end{lemma}
\begin{proof}
By Remark~\ref{mod} and Theorem~\ref{finitedim} it suffices to
consider isotropy groups of solutions $(z,\lambda)$ of the equation
\begin{equation}{\label{fa}}
\nabla_xa(z,\lambda)=0,
\end{equation}
such that $z\in E_f\backslash\set{x_0}$ (such
solutions are regarded as nontrivial solutions of~\eqref{fa}). Since
$\nabla_x^2a(x_0,\lambda_0)$ is symmetric, one can use the
decomposition
\begin{equation*}
E_f=(\im\nabla_x^2a(x_0,\lambda_0))
\oplus(\ker\nabla_x^2a(x_0,\lambda_0))
\end{equation*}
and write~\eqref{fa} as the system of equations
\begin{align}
\Pi \nabla_xa(u,(v,\lambda))&=0 \label{faim}\\
\nonumber (Id-\Pi) \nabla_xa(u,(v,\lambda))&=0,
\end{align}
where $\Pi$ is a~projection of $E_f$ onto
$\im\nabla_x^2a(x_0,\lambda_0),$ $u=\Pi(z),$ $v=(Id-\Pi)(z).$ Write
also $x_0=(u_0,v_0),$ where $u_0=\Pi(x_0),$
$v_0=(Id-\Pi)(x_0).$ Applying the \mbox{$\SO(2)$-equi}variant
version of the implicit function theorem to~\eqref{faim} one obtains
the existence of an open neighbourhood $W$ of
$u_0\in\im\nabla_x^2a(x_0,\lambda_0),$ an open neighbourhood $V$ of
$(v_0,\lambda_0)\in\ker\nabla_x^2a(x_0,\lambda_0)\times\R^k,$ and an
\mbox{$\SO(2)$-equi}variant mapping $\gamma\colon V\rightarrow W$
of class $C^{1,0}$ such that if
$(u,(v,\lambda))\in W\times V$ is a~solution of~\eqref{fa} then
$u=\gamma(v,\lambda).$ (In particular, $\gamma(v_0,\lambda)=u_0$
for every $\lambda\in\R^k$ such that $(v_0,\lambda)\in V,$ since
$\nabla_x a(u_0,(v_0,\lambda))=0.$) Thus all nontrivial solutions
of~\eqref{fa} in $U\eqdf W\times V$ are of the form
$(\gamma(v,\lambda),(v,\lambda)),$ where
$v\in\ker\nabla_x^2a(x_0,\lambda_0),$ $v\neq v_0.$ Their isotropy
groups are equal to
\begin{equation*}
\SO(2)_{(\gamma(v,\lambda),(v,\lambda))}
=\SO(2)_{\gamma(v,\lambda)}\cap\SO(2)_{(v,\lambda)}.
\end{equation*}
Furthermore,
$\SO(2)_{(v,\lambda)}\subset\SO(2)_{\gamma(v,\lambda)},$ since
the mapping $\gamma$ is \mbox{$\SO(2)$-equi}\-va\-riant, hence
\begin{equation*}
\SO(2)_{(\gamma(v,\lambda),(v,\lambda))}
=\SO(2)_{(v,\lambda)}=\SO(2)_{v},
\end{equation*}
where $v\in\ker\nabla_x^2a(x_0,\lambda_0),$ $v\neq v_0.$ Thus if
$v_0=0$ then the isotropy groups of nontrivial solutions
of~\eqref{fa} in a~neighbourhood of $(x_0,\lambda_0)$ belong to the
set of isotropy groups of nonzero elements of
$\ker\nabla_x^2a(x_0,\lambda_0).$ The same condition is obtained
for $v_0\neq 0$ by choosing the set $V$ in such a~way that
$(0,\lambda)\not\in V$ for all $\lambda\in\R^k.$ Finally, observe
that, by Theorem~\ref{finitedim},
$\ker\nabla_x^2a(x_0,\lambda_0)\subset E(\lambda_0),$ and that
$\dim E(\lambda_0)<4nr_0+2n<\infty,$ since
$E(\lambda_0)\subset E_f.$
\end{proof}

The set $G(\lambda_0)$ from Lemma~\ref{isotropy} consists of
groups $\K_j,$ $j\in X(\lambda_0),$ and all their
intersections. Every such intersection is also equal to $\K_l$ for
some $l\in\N\cup\set{0}.$ Namely, $\Z_j\cap\SO(2)=\Z_j$ and
$\Z_j\cap\Z_m=\Z_l,$ where $l$ is the greatest common divisor of
$j$ and $m.$ Thus
\begin{equation*}
G(\lambda_0) \subset\set{\K_j\cond
j\in\set{0,\ldots,\max(X(\lambda_0))}}.
\end{equation*}
In particular, the set $G(\lambda_0)$ is finite, since
$X(\lambda_0)$ is finite.

As a~consequence of Lemma~\ref{isotropy} and the definition of
\mbox{$j$-so}lution one obtains the following version of necessary
conditions for bifurcation formulated in~\cite{DR}.

\begin{corollary}{\label{necniel}}
Let $H\in C^{2,0}(\R^{2n}\times\R^k,\R).$ If
$(x_0,\lambda_0)\in\cT(x_0)$ is a~bifurcation point of nontrivial
solutions of~\eqref{parhamniel} then
$\lambda_0\in\Lambda_0(\nabla_x^2H(x_0,\cdot))\cup
\Lambda(\nabla_x^2H(x_0,\cdot)).$ Namely,
\begin{enumerate}
\item if $(x_0,\lambda_0)$ is a~bifurcation point of nontrivial
      stationary solutions then
      \begin{equation*}
      \lambda_0\in\Lambda_0(\nabla_x^2H(x_0,\cdot));
      \end{equation*}
\item{\label{necjsol}} if $(x_0,\lambda_0)$ is a~bifurcation
      point of nonstationary
      \mbox{$j$-solu}tions for some $j\in\N$ then
      \begin{equation*}
      \lambda_0\in
      \bigcup_{l\in\N}\Lambda_{lj}(\nabla_x^2H(x_0,\cdot)).
      \end{equation*}
\end{enumerate}
\end{corollary}
\begin{proof}
If $\lambda_0\not\in\Lambda_0(\nabla_x^2H(x_0,\cdot))$ then there
are no nontrivial stationary solutions in a~neighbourhood of
$(x_0,\lambda_0),$ according to Lemma~\ref{isotropy}.
To prove~\eqref{necjsol}
assume that $U,$ $G(\lambda_0),$ and $X(\lambda_0)$ are such as in
Lemma~\ref{isotropy}. The isotropy group of every nonstationary
\mbox{$j$-solu}tion from the set $U$ contains $\Z_j,$ therefore it
is equal to $\Z_{sj}$ for some $s\in\N,$ which depends on the
solution. The group $\Z_{sj}$ belongs to
$G(\lambda_0)\backslash\set{\SO(2)},$ according to
Lemma~\ref{isotropy}. Thus $\Z_{sj}\subset\Z_{r}$ for
some $r\in X(\lambda_0)\backslash\set{0},$ which implies that
$r=msj$ for some $m\in\N.$ Setting $l=ms$ one has
$\det Q_{lj}(\nabla_x^2H(x_0,\lambda_0))
=\det Q_r(\nabla_x^2H(x_0,\lambda_0))=0,$ hence
$\lambda_0\in\Lambda_{lj}(\nabla_x^2H(x_0,\cdot)).$
\end{proof}

\begin{remark}{\label{remnec}}
If $(x_0,\lambda_0)$ is completely degenerate,
i.e. $\nabla_x^2 H(x_0,\lambda_0)=0,$ then for every $j\in\N$ one
has $\det Q_j(\nabla_x^2 H(x_0,\lambda_0))\neq 0.$ In such
a~case $(x_0,\lambda_0)$ is not a~bifurcation point of nonstationary
solutions of~\eqref{parhamniel}, in view of Corollary~\ref{necniel}.
\end{remark}

As a~consequence it is now proved that if a~completely degenerate
stationary point of a~Hamiltonian system without parameter is an
emanation point of periodic orbits then the minimal periods of that
orbits tend to infinity as the orbits converge to the
stationary point.

\begin{corollary}{\label{infpercor}}
Let $H\in C^2(\R^{2n},\R).$ If $x_0\in (\nabla H)^{-1}(\set{0})$ is
completely degenerate, i.e. $\nabla^2H(x_0)=0,$ then for every $C>0$
there exists $\delta>0$ such that every nonstationary periodic orbit
of~\eqref{ham} contained in the ball in $\R^{2n}$
centred at $x_0$ with radius $\delta$ has the minimal period greater
then $C.$
\end{corollary}
\begin{proof}
Let $\set{x_n}_{n\in\N}$ be a~sequence of nonstationary periodic
solutions of~\eqref{ham} such that $\norm{x_n-x_0}_{0}<\frac{1}{n}$
for every $n\in\N,$ where $\norm{\cdot}_{0}$ denotes the supremum
norm. Let $T_n$ be the minimal period of $x_n$ for every $n\in\N.$
Set $\overline{x}_n(t)\eqdf x_n((T_n/2\pi)t)$ and
$\lambda_n\eqdf T_n/2\pi.$
Then $(\overline{x}_n,\lambda_n)$ is a~solution of~\eqref{parham}.
Suppose, on the contrary, that the sequence
$\set{T_n}_{n\in\N}$ has a~bounded subsequence. Then passing to
a~subsequence once again one can assume that
$\lambda_n\to\lambda_0$ as $n\to\infty$ for some $\lambda_0\in\R.$
Now, let
\begin{equation*}
M_n=\sup_{\abs{\xi-x_0}\leq\frac{1}{n}}\abs{J\nabla H(\xi)}^2.
\end{equation*}
Notice that $M_n\to 0$ as $n\to\infty,$ since $\nabla H(x_0)=0.$
As in estimate~(5.1) in~\cite{Rd} one has
\begin{align*}
\norm{\overline{x}_n-x_0}_{\hilb}^2&=
\int_{0}^{2\pi}\left[\abs{\overline{x}_n(t)-x_0}^2
+ \abs{\dot{\overline{x}}_n(t)}^2\right]\mathrm{d}t\\
&\leq 2\pi\norm{\overline{x}_n-x_0}_{0}^2
+\lambda_n^2\int_{0}^{2\pi}
\abs{J\nabla H(\overline{x}_n(t))}^2\mathrm{d}t\\
&\leq 2\pi\norm{\overline{x}_n-x_0}_0^2+\lambda_n^2 2\pi M_n
<2\pi\frac{1}{n^2}+\lambda_n^2 2\pi M_n,
\end{align*}
since
$\norm{\overline{x}_n-x_0}_{0}=\norm{x_n-x_0}_{0}<\frac{1}{n}.$ Thus
$(\overline{x}_n,\lambda_n)\to (x_0,\lambda_0)$ as $n\to\infty$ in
$\hilb\times\R,$ a~contradiction (see Remark~\ref{remnec}).
\end{proof}

\begin{remark}
Lemma~\ref{isotropy} can exclude symmetry breaking in the situation
when Corollary~\ref{necniel} does not exclude it. For example,
assume that
$\det Q_6(\nabla_x^2H(x_0,\lambda_0))=0$ and
$\det Q_j(\nabla_x^2H(x_0,\lambda_0))\neq 0$ for $j\in\N,$
$j\neq 6,$ which holds for diagonal matrix
$\nabla_x^2H(x_0,\lambda_0)$ with the $n$th and the $2n$th element
of the diagonal equal to $6$ and the rest of elements equal to $0$
(see Remark~\ref{inters} and condition~\eqref{maxrem}). Then
$E(\lambda_0)=E_0\cup E_6$ and the only possible isotropy group of
nonstationary solutions of~\eqref{parhamniel} in a~neighbourhood of
$(x_0,\lambda_0)$ is $\Z_6,$ which corresponds to the minimal period
$\frac{2\pi}{6}.$ This excludes symmetry breaking if
$(x_0,\lambda_0)$ is not a~bifurcation point of nontrivial
stationary solutions. However, Corollary~\ref{necniel} does not
exclude bifurcation of solutions of~\eqref{parhamniel} with the
minimal period $\frac{2\pi}{3}$ from $(x_0,\lambda_0),$ since
$\displaystyle \lambda_0\in\Lambda_6
=\Lambda_{2\cdot 3}\subset\bigcup_{l\in\N}\Lambda_{l\cdot 3}.$
\end{remark}

\section{Dancer-Rybicki bifurcation theorem for
\mbox{$j$-so}lutions.}

In this section global bifurcation theorems for
\mbox{$j$-so}lutions of~\eqref{parhamniel} are proved
in the case of systems with one
parameter ($k=1$), i.e. it is assumed that
$H\in C^{2,0}(\R^{2n}\times\R,\R).$

Let $x_0\in\R^{2n}$ satisfy~\eqref{stationary} for $k=1$ and fix
$\lambda_0\in\R.$ Assume that for sufficiently small $\varepsilon>0$
and every $\lambda\in
[\lambda_0-\varepsilon,\lambda_0+\varepsilon]\backslash\set{\lambda_0}$
there exists a~neighbourhood $W\subset\R^{2n}\times\R$ of
$(x_0,\lambda)$ such that
$(\nabla_xH)^{-1}(\set{0})\cap W \subset\set{x_0}\times\R.$
Set
\begin{equation*}
\eta_0(x_0,\lambda_0)
=\ind{\nabla_x H(\cdot,\lambda_0+\varepsilon)}{x_0}
-\ind{\nabla_x H(\cdot,\lambda_0-\varepsilon)}{x_0}.
\end{equation*}
If $\lambda_0\in\R$ is not a~cluster point of the set
$\Lambda_j(\nabla_x^2H(x_0,\cdot))$ for some $j\in\N$ then one can
choose  $\varepsilon>0$ such that
$\Lambda_j(\nabla_x^2H(x_0,\cdot)) \cap
[\lambda_0-\varepsilon,\lambda_0+\varepsilon]=\set{\lambda_0}$ and
set
\begin{align*}
\eta_j(x_0,\lambda_0) &
=\ind{\nabla_x H(\cdot,\lambda_0+\varepsilon)}{x_0}
\cdot
\frac{\morse{Q_j(\nabla^2_xH(x_0,\lambda_0+\varepsilon))}}{2} \\
&\quad -\ind{\nabla_x H(\cdot,\lambda_0-\varepsilon)}{x_0}
\cdot
\frac{\morse{Q_j(\nabla^2_xH(x_0,\lambda_0-\varepsilon))}}{2}.
\end{align*}
The sequence
\begin{equation*}
\eta(x_0,\lambda_0)=\set{\eta_j(x_0,\lambda_0)}_{j=0}^{\infty}
\end{equation*}
is called a~\emph{bifurcation index of $(x_0,\lambda_0)$}. Usually
only selected coordinates of this index are needed. Notice that
infinitely many of them may be nonzero. However, according to
Theorem~\ref{finitedim} and Remark~\ref{mod} there exists $r_0\in\N$
such that
$\eta_j(x_0,\lambda_0)=\eta_0(x_0,\lambda_0)\cdot n$
for $j>r_0.$
In the proof of Theorem~\ref{genbif} some coordinates of
$\eta(x_0,\lambda_0)$ will be identified with coordinates of the
index $\IND(x_0,\lambda_0)$ defined by~\eqref{abstrind} for an
appropriate mapping $\nabla_xf.$

Consider first the case of system~\eqref{parham} with linear
dependence on parameter, which can be written in
form~\eqref{parhamniel} for $H$ replaced by
$\hat{H}\in C^2(\R^{2n}\times\R,\R)$ defined by
\begin{equation*}
\hat{H}(x,\lambda)=\lambda H(x).
\end{equation*}
To define $\eta$ in this case it suffices to assume that $x_0$ is an
isolated element of $(\nabla H)^{-1}(\set{0}).$
(For every $a,b\in\R,$ $a<b,$ the set
$\Lambda(\nabla_x^2\hat{H}(x_0,\cdot))\cap [a,b]$ is finite.)
Then
\begin{align*}
\eta_0(x_0,\lambda_0) &= 0, \\
\eta_j(x_0,\lambda_0) &=  \ind{\nabla H}{x_0}\cdot
\left(\frac{\morse{Q_j((\lambda_0+\varepsilon)\nabla^2H(x_0))}}{2}
\right. \\
&\qquad\qquad\qquad\quad \left.
-\frac{\morse{Q_j((\lambda_0-\varepsilon)\nabla^2H(x_0))}}{2}\right)
\end{align*}
for every $j\in\N,$ as in~\cite{DR}. Notice that
$\eta_j(x_0,\lambda_0)=0$ for every $j>r_0,$
hence $\eta(x_0,\lambda_0)\in U(\SO(2)).$

As it was proved in~\cite{DR}, for every $K>0$ there exists
$\delta>0$ such that every solution $(x,\lambda)\in\hilb\times\R$
of~\eqref{parham} satisfying the conditions
$\abs{\lambda}\leq\delta$ and $\norm{x}_{0}\leq K$ is stationary.
(In particular, $(x_0,0)$ is not a~bifurcation point of
nonstationary solutions of~\eqref{parham}.) Thus it suffices to
consider the solutions of~\eqref{parham} for $\lambda>0.$

The set $\cT=(\nabla H)^{-1}(\set{0})\times\halfline$ is
regarded as the set of trivial solutions of~\eqref{parham}
and nontrivial solutions are the nonstationary ones.
If $(x_0,\lambda_0)$ is
a~global bifurcation point of nonstationary solutions
of~\eqref{parham} then
$C(x_0,\lambda_0)$ denotes the connected component of the
closure of the set of nonstationary solutions of~\eqref{parham}
containing $(x_0,\lambda_0).$

The following Rabinowitz-type global bifurcation theorem for
Hamiltonian systems has been proved by Dancer and Rybicki~\cite{DR}.

\begin{theorem}{\label{globalprel}}
Let $H\in C^2(\R^{2n},\R)$ and let $(\nabla H)^{-1}(\set{0})$ be
finite. Fix $x_0\in(\nabla H)^{-1}(\set{0})$ and
$\lambda_0\in\Lambda^+(\nabla_x^2\hat{H}(x_0,\cdot)).$ If
$\eta(x_0,\lambda_0)\neq\Theta$ then $(x_0,\lambda_0)$ is a~global
bifurcation point of nonstationary solutions of~\eqref{parham}.
Moreover, if the set $C(x_0,\lambda_0)$ is bounded in
$\hilb\times\halfline$ then
$C(x_0,\lambda_0)\cap\cT=\set{y_1,\ldots,y_m}$ for some $m\in\N,$
$y_1,\ldots,y_m\in\cT$ such that
\begin{equation*}
\sum_{i=1}^{m}\eta(y_i)=\Theta.
\end{equation*}
\end{theorem}

In this section generalized versions of Theorem~\ref{globalprel}
concerning \mbox{$j$-solu}tions (for systems with nonlinear
dependence on parameter) are proved. To this aim, the method
presented in~\cite{DR} is applied to the restriction of the mapping
$\nabla_xa$ from Theorem~\ref{finitedim} to the subspace of fixed
points of the action of the group $\K_j$ for given
$j\in\N\cup\set{0}.$

Consider an orthogonal representation of the group $\SO(2)$ on
a~real inner product space $V$ with $\dim V<\infty$ and let
$\nabla_x f\colon V\times\R\rightarrow V$ be a~continuous
\mbox{$\SO(2)$-equi}variant gradient mapping. Let
$\Delta\times\R\subset (\nabla_x f)^{-1}(\set{0})$ be the set of
trivial solutions of the equation
\begin{equation}{\label{abstreq}}
\nabla_x f(x,\lambda)=0
\end{equation}
for some finite set $\Delta\subset V.$
If $(x_0,\lambda_0)\in \Delta\times\R$ is a~branching point of
nontrivial solutions of~\eqref{abstreq} then
$\Sigma(x_0,\lambda_0)$ denotes the connected component of the
closure of the set of nontrivial solutions of~\eqref{abstreq}
containing $(x_0,\lambda_0).$

The following theorem is a~slightly modified version of Theorem~2.2
formulated in~\cite{DR}. It will be used in the proof
of Theorem~\ref{genbif}.

\begin{theorem}{\label{abstrth}}
Let $\Omega$ be a~bounded open subset of $V\times\R.$ Assume that
$(\Delta\times\R)\cap \Omega$ contains at most finite number of
bifurcation points of nontrivial solutions of~\eqref{abstreq} and
$\partial\Omega$ contains no bifurcation points. If
$\IND(x_0,\lambda_0)\neq \Theta$ for some
$(x_0,\lambda_0)\in (\Delta\times\R)\cap \Omega$
then $(x_0,\lambda_0)$ is a~branching
point of nontrivial solutions of~\eqref{abstreq}. Moreover, if
$\Sigma(x_0,\lambda_0)\cap \partial \Omega=\emptyset$ then
$\Sigma(x_0,\lambda_0)\cap(\Delta\times\R)\cap\Omega
=\set{z_1,\ldots,z_m}$
for some $m\in\N,$ $z_1,\ldots,z_m\in\Delta\times\R$ such that
\begin{equation*}
\sum_{i=1}^{m}\IND(z_i)=\Theta.
\end{equation*}
\end{theorem}

The proof of the above theorem proceeds analogously to that of
Theorem~29.1 in~\cite{Dei}. (It is based on Whyburn lemma and
standard properties of topological degree.) The Brouwer degree
($\dim V<\infty$) is replaced by the degree $\DEG$ in this case.
To guarantee that sets over which the degree
$\DEG$ is computed are \mbox{$\SO(2)$-equi}variant it suffices to
observe that if $D\subset V\times\R$ then the set
$\SO(2)D\eqdf\set{gv\cond g\in\SO(2),v\in D}$ is
\mbox{$\SO(2)$-equi}variant and
$(\nabla_x f)^{-1}(\set{0})\cap D
=(\nabla_x f)^{-1}(\set{0})\cap \SO(2)D,$
since the mapping $\nabla_x f$ is \mbox{$\SO(2)$-equi}variant.

Assume that $H\in C^{2,0}(\R^{2n}\times\R,\R)$ and
let $\Delta\times\R\subset(\nabla_x H)^{-1}(\set{0})$ be the
set of trivial solutions of~\eqref{parhamniel} for some finite
set $\Delta\subset\R^{2n}.$ (Notice that some
nontrivial solutions can be stationary.) Set
\begin{equation*}
P_j(\Delta)= \bigcup_{x_0\in\Delta}
\left(\set{x_0}\times
\bigcup_{l\in\N}\Lambda_{lj}(\nabla_x^2H(x_0,\cdot))\right)
\end{equation*}
for $j\in\N\cup\set{0}.$ For a~fixed bounded open set
$U\subset\hilb\times\R$ use will be made of the following condition.
\begin{trivlist}
\item[\textbf{(N)}] For every
           $(x,\lambda)\in(\Delta\times\R)\cap\cl{U}\backslash
           P_j(\Delta)$
           there exists its neighbourhood
           $W\subset\R^{2n}\times\R$ such that
           $(\nabla_xH)^{-1}(\set{0})\cap W
           \subset\Delta\times\R.$
\end{trivlist}
If $(x_0,\lambda_0)\in P_j(\Delta)$ is a~branching point
of nontrivial \mbox{$j$-solu}tions of~\eqref{parhamniel} for some
$j\in\N\cup\set{0},$ then $K_j(x_0,\lambda_0)$ denotes
the connected component of the closure of the set of nontrivial
(possibly stationary) \mbox{$j$-solu}tions containing
$(x_0,\lambda_0).$

\begin{theorem}{\label{genbif}}
Let $H\in C^{2,0}(\R^{2n}\times\R,\R).$
Fix $j\in\N\cup\set{0}$ and let $U$ be a~bounded open subset of
$\hilb\times\R.$ Assume that the set $P_j(\Delta)\cap U$ is finite,
$P_j(\Delta)\cap \partial U=\emptyset,$ and condition (N) is
satisfied. If $\eta_{j}(x_0,\lambda_0)\neq 0$ for some
$(x_0,\lambda_0)\in P_j(\Delta)\cap U$ then $(x_0,\lambda_0)$ is
a~branching point of nontrivial (possibly stationary)
\mbox{$j$-solu}tions of~\eqref{parhamniel}. Moreover, if
$K_j(x_0,\lambda_0)\cap \partial U =\emptyset$ then
$K_j(x_0,\lambda_0)\cap(\Delta\times\R)\cap U=\set{z_1,\ldots,z_m}$
for some $m\in\N,$ $z_1,\ldots,z_m\in\Delta\times\R$ such that
\begin{equation*}
\sum_{i=1}^{m}\eta_{lj}(z_i)=0 \qquad
\text{for every $l\in\N\cup\set{0}.$}
\end{equation*}
\end{theorem}
\begin{proof}
$H$ can be replaced by $H_1$ from Remark~\ref{mod}. One has
$(\Delta\times\R)\cap\cl{U}\subset B(0,\eta)$ and the functions $H$
and $H_1$ are equal on $B(0,\eta).$ The solutions
of~\eqref{parhamniel} in $\cl{U}$ are those of~\eqref{parhammod}.
In view of Theorem~\ref{finitedim}, $(x_0,\lambda_0)$ is a~branching
point of nontrivial \mbox{$j$-solu}tions of~\eqref{parhammod} iff it
is a~branching point of nontrivial \mbox{$j$-solu}tions of the
equation $\nabla_x a(x,\lambda)=0$ in the space $E_f\times\R,$ which
means that $(x_0,\lambda_0)$ is a~branching point of nontrivial
solutions of~\eqref{abstreq} in the space $V\times\R,$ where
$\displaystyle V
=(E_f)^{\K_j}=E_f\cap\bigoplus_{l=0}^{\infty}E_{lj}$
and
$\nabla_x f=(\nabla_x a)|_{(V\times\R,V)}.$
($E_f$ can be regarded as a~subspace of $\hilb.$) Notice that the
only solutions of~\eqref{abstreq} in $V\times\R$ are then
\mbox{$j$-solu}tions. The set of trivial solutions and bifurcation
points of \mbox{$j$-solu}tions remain the same as in the case
of~\eqref{parhammod}. Use will be made of Theorem~\ref{abstrth}.
According to Lemma~\ref{abstrlem}, for every
$k\in\N\cup\set{0}$ one has
\begin{equation*}
\IND_k(x_0,\lambda_0)=
\begin{cases}
\eta_k(x_0,\lambda_0) &
\text{if $k=lj\leq r_0$ for some $l\in\N\cup\set{0}$},\\
0 & \text{otherwise},
\end{cases}
\end{equation*}
where $\IND(x_0,\lambda_0)$ is the bifurcation index
\eqref{abstrind}.
Notice also that if $j>r_0$ then
\begin{equation*}
\eta_j(x_0,\lambda_0)=\eta_0(x_0,\lambda_0)\cdot n
=\IND_0(x_0,\lambda_0)\cdot n.
\end{equation*}
(In this case, in view of Theorem~\ref{abstrth},
$(x_0,\lambda_0)$ is a~branching point of nontrivial stationary
solutions, which are \mbox{$j$-solu}tions for every
$j\in\N\cup\set{0}$). Furthermore, in view of the assumptions and
Corollary~\ref{necniel}, $U$ contains at most finite number of
bifurcation points of nontrivial \mbox{$j$-solu}tions
of~\eqref{parhammod} and if $(x_0,\lambda_0)\in U$ is a~branching
point of nontrivial \mbox{$j$-solu}tions
of~\eqref{parhammod} such that
$K_j(x_0,\lambda_0)\cap\partial U=\emptyset$ then
$K_j(x_0,\lambda_0)$ is compact (see Lemma~\ref{compbranch}),
so also is
$\Sigma(x_0,\lambda_0)\eqdf h^{-1}(K_j(x_0,\lambda_0)),$
where $h$ is the homeomorphism from Theorem~\ref{finitedim}.
Thus one can find $\Omega\subset V\times\R$ satisfying the
assumptions of Theorem~\ref{abstrth} and the condition
$\Sigma(x_0,\lambda_0)\subset\Omega.$
\end{proof}

As a~consequence of Theorem~\ref{genbif} one obtains the next
two theorems that will be used in subsequent sections. In the first
one it is assumed that $(x_0,\lambda_0)$ is not a~bifurcation point
of nontrivial stationary solutions, whereas in the second one
such bifurcation is allowed but it is assumed that $\lambda_0$
is not a~cluster point of $\Lambda_0(\nabla_x^2H(x_0,\cdot)),$
which means that all the points from
$\cT(x_0)\backslash\set{(x_0,\lambda_0)}$
in a~neighbourhood of $(x_0,\lambda_0)$ are nondegenerate
(although $(x_0,\lambda_0)$ can be degenerate). In both cases
$\cT(x_0)=\set{x_0}\times\R$ is regarded as the set of trivial
solutions, i.e. $\Delta=\set{x_0}.$

\begin{theorem}{\label{genzdeg}}
Let $H\in C^{2,0}(\R^{2n}\times\R,\R)$ and
$\nabla_xH(x_0,\lambda)=0$ for some $x_0\in\R^{2n}$ and all
$\lambda\in\R.$ Assume that $\lambda_0$ is an isolated element of
the set
$\displaystyle
\bigcup_{l\in\N}\Lambda_{lj}(\nabla_x^2H(x_0,\cdot))$
for some
$j\in\N$ and $(x_0,\lambda_0)$ is not a~bifurcation point of
nontrivial stationary solutions of~\eqref{parhamniel}. If
$\eta_{j}(x_0,\lambda_0)\neq
0$ then $(x_0,\lambda_0)$ is a~branching point of
nonstationary \mbox{$j$-solu}tions of~\eqref{parhamniel}
and a~global bifurcation point of nontrivial \mbox{$j$-solu}tions.
\end{theorem}

\begin{theorem}{\label{genniezd}}
Let $H\in C^{2,0}(\R^{2n}\times\R,\R)$ and
$\nabla_xH(x_0,\lambda)=0$ for some $x_0\in\R^{2n}$ and all
$\lambda\in\R.$ Assume that $\lambda_0$ is an isolated element of
$\displaystyle\Lambda_0(\nabla_x^2H(x_0,\cdot))
\cup\bigcup_{l\in\N}\Lambda_{lj}(\nabla_x^2H(x_0,\cdot))$ for some
$j\in\N\cup\set{0}.$ If $\eta_{j}(x_0,\lambda_0)\neq 0$ then
$(x_0,\lambda_0)$ is a~global bifurcation point of nontrivial
(possibly stationary) \mbox{$j$-solu}tions of~\eqref{parhamniel}.
\end{theorem}

Recall that in the case of system~\eqref{parham} with linear
dependence on parameter the set
$\cT=(\nabla H)^{-1}(\set{0})\times\halfline$ is regarded as
the set of trivial solutions. If $(x_0,\lambda_0)$ is a global bifurcation
point of nonstationary \mbox{$j$-so}lutions
of~\eqref{parham} then $C_j(x_0,\lambda_0)$ denotes the
connected component of the closure of the set of nonstationary
\mbox{$j$-so}lutions of~\eqref{parham} containing
$(x_0,\lambda_0).$
Theorem~\ref{genbif}
implies the following generalized version
of Theorem~\ref{globalprel}.

\begin{theorem}{\label{global}}
Let $H\in C^2(\R^{2n},\R)$ and let $(\nabla H)^{-1}(\set{0})$ be
finite. Fix $x_0\in(\nabla H)^{-1}(\set{0})$ and
$\lambda_0\in\Lambda^+(\nabla_x^2\hat{H}(x_0,\cdot)).$ If
$\eta_j(x_0,\lambda_0)\neq 0$ for some $j\in\N$ then
$(x_0,\lambda_0)$ is a~global bifurcation point of nonstationary
\mbox{$j$-solu}tions of~\eqref{parham}. Moreover, if the set
$C_j(x_0,\lambda_0)$ is bounded in $\hilb\times\halfline$ then
$C_j(x_0,\lambda_0)\cap\cT=\set{z_1,\ldots,z_m}$ for some $m\in\N,$
$z_1,\ldots,z_m\in\cT$ such that
\begin{equation*}
\sum_{i=1}^{m}\eta_{lj}(z_i)=0 \qquad \text{for every $l\in\N.$}
\end{equation*}
\end{theorem}

The conclusion of Theorem~\ref{global} does not seem to follow from
Theorem~\ref{globalprel}, since the formula for the sum of
bifurcation indices over the branch $C(x_0,\lambda_0)$ does not
imply the formula for the sum of indices over the branch
$C_j(x_0,\lambda_0)\subset C(x_0,\lambda_0).$

The results from~\cite{Rd,RR} provide sufficient conditions for
global bifurcation of (\mbox{$2\pi$-pe}\-riodic) solutions
of~\eqref{parham} and describe unbounded branches of solutions
bifurcating from given points. Theorem~\ref{global} allows to
replace that branches by appropriate branches of
\mbox{$j$-solu}tions. For example, taking into account
Corollary~\ref{necniel} and Theorem~\ref{global} one can generalize
Lemma~3.3, Theorem~4.6, and Corollary~5.3 from~\cite{RR} as follows.

\begin{corollary}
Assume that $H\in C^2(\R^{2n},\R)$ and that
$(\nabla H)^{-1}(\set{0})$
is finite. Let $x_0\in(\nabla H)^{-1}(\set{0})$ be such that
$\ind{\nabla H}{x_0}\neq 0$ and
$\nabla^2H(x_0) =\diag{A,B},$ $A,B\in\Sym(n,\R),$
where $A$ or $B$ is strictly positive or strictly negative
definite. Then for every $j\in\N$ the set of bifurcation points
$(x_0,\lambda)\in\set{x_0}\times\halfline$ of nonstationary
\mbox{$j$-solu}tions of~\eqref{parham} is equal to the set of
global bifurcation points of nonstationary \mbox{$j$-solu}tions
and equal to
\begin{equation*}
\set{\left(x_0,\frac{lj}{\sqrt{\nu}}\right)\cond
\nu\in\pspect{AB},l\in\N}.
\end{equation*}
Furthermore, if $(\nabla H)^{-1}(\set{0})=\set{x_0},$
$\lambda_0=\frac{j_0}{\sqrt{\nu_0}},$ $j_0\in\N,$
$\nu_0\in\pspect{AB}$ then the set $C_{j_0}(x_0,\lambda_0)$ is
unbounded in $\hilb\times\halfline.$
\end{corollary}

In the above corollary the set $C(x_0,\lambda_0)$ from
Corollary~5.3 in~\cite{RR} has been replaced by
$C_{j_0}(x_0,\lambda_0).$ Similarly, in the case when
$(\nabla H)^{-1}(\set{0})$ is not a~singleton but it is
finite, the unbounded branches $C(\xi,\frac{j}{\sqrt{\omega}})$ of
\mbox{$2\pi$-pe}riodic solutions in Corollaries~5.5-5.7 from
\cite{RR} can be replaced by the unbounded
branches~$C_j(\xi,\frac{j}{\sqrt{\omega}})$ of \mbox{$j$-solu}tions.

\section{Global bifurcation points in multiparameter
systems}{\label{gen}}

In this section global bifurcation and symmetry breaking
theorems for system~\eqref{parhamniel} are proved in the case of
arbitrary number $k$ of parameters. To this aim use is made of the
bifurcation theorems for the system with one parameter obtained in
the previous section.

\begin{definition}
Let $H\in C^{2,0}(\R^{2n}\times\R^k,\R),$ $j\in\N,$ and fix
$x_0\in\R^{2n}.$
A~function $F_j\in C(\R^k,\R)$
is called a~\emph{\mbox{$j$th} detecting function
for system~\eqref{parhamniel}} provided that the following
conditions are satisfied.
\begin{enumerate}
\item For every $\lambda\in\R^k,$ $F_j(\lambda)=0$ iff
$\det Q_j(\nabla_x^2H(x_0,\lambda))=0.$
\item For every straight line $L\subset \R^k$ and every
$\lambda_0\in L$ being an isolated zero of $F_j$ on $L,$
if $F_j$ changes its sign on $L$ at $\lambda_0$ then the
function
$\R^k\ni\lambda\mapsto \morse{Q_j(\nabla_x^2H(x_0,\lambda))}$
changes its value on $L$ at $\lambda_0.$
\end{enumerate}
$F_0\eqdf \det Q_0(\nabla_x^2H(x_0,\cdot))$ is called the
\emph{\mbox{$0$th} detecting function for~\eqref{parhamniel}}.
$\set{F_j}_{j\in\N\cup\set{0}}$ is said to be
a~\emph{detecting sequence for~\eqref{parhamniel}}
if $F_j$ is a~\mbox{$j$th} detecting
function for~\eqref{parhamniel} for every $j\in\N\cup\set{0}.$
\end{definition}

\begin{remark}{\label{biffun}}
Fix $x_0\in\R^{2n}$ and $j\in\N.$ Since $H$ is of class $C^{2,0}$
and for every $\lambda\in\R^k$ each eigenvalue $\nu(\lambda)$
of the symmetric matrix $Q_j(\nabla_x^2H(x_0,\lambda))$ has even
multiplicity $\mult{\nu(\lambda)}$ (see Remark~\ref{qjeven}),
there exist
$\nu_1,\ldots,\nu_{2n}\in C(\R^k,\R)$ such that
$\spect{Q_j(\nabla_x^2H(x_0,\lambda))}
=\set{\nu_1(\lambda),\ldots,\nu_{2n}(\lambda)}$
and for every $i\in\set{1,\ldots,2n}$ the eigenvalue
$\nu_i(\lambda)$ occurs
$\frac{1}{2}\mult{\nu_i(\lambda)}$ times in the \mbox{$2n$-tuple}
$(\nu_1(\lambda),\ldots,\nu_{2n}(\lambda)).$
Then the function $F_j$ defined by
$F_j(\lambda)=\nu_1(\lambda)\cdot\ldots\cdot\nu_{2n}(\lambda)$
is a~\mbox{$j$th} detecting function for~\eqref{parhamniel}.
Notice that the mapping
$\R^k\ni \lambda \mapsto \det Q_j(\nabla_x^2(x_0,\lambda))$
is nonnegative for every $j\in\N,$ therefore it cannot be used to
detect the change of the Morse index of
$Q_j(\nabla_x^2(x_0,\lambda)).$
\end{remark}

Now, explicit formulae for detecting functions (exploited in
examples in Section~\ref{examples}) will be given
in the case when
\begin{equation}{\label{bldi}}
\forall_{\lambda\in\R^k}:\;\;\nabla_x^2H(x_0,\lambda)
=\left[
\begin{array}{cc}
A(\lambda)&0\\
0&B(\lambda)
\end{array}
\right],\;\;\;A(\lambda),B(\lambda)\in\Sym(n,\R).
\end{equation}

If $C,D\in\Sym(n,\R)$ and $K\in\Sym(2n,\R)$ is of the form
\begin{equation*}
K= \left[
\begin{array}{cc}
C&0\\
0&D
\end{array}
\right]
\end{equation*}
then for every $j\in\N\cup\set{0}$ define $G_j(K)\in\Sym(2n,\R)$
and $X\in\Ort(4n,\R)$ as follows.
\begin{equation*}
G_j(K)= \left[
\begin{array}{cc}
-C&jI\\
jI&-D
\end{array}
\right] = -K+j\left[
\begin{array}{cc}
0&I\\
I&0
\end{array}
\right],
\end{equation*}
\begin{equation*}
X= \left[
\begin{array}{cccc}
I&0&0&0\\
0&0&0&-I\\
0&0&I&0\\
0&I&0&0
\end{array}
\right],
\end{equation*}
where $I\equiv I_n.$

\begin{lemma}[\cite{RR}]{\label{det}}
For every $j\in\N$ one has
\begin{enumerate}
\item
\begin{equation*}
X^t Q_j(K) X= \frac{1}{1+j^2}\left[
\begin{array}{cc}
G_j(K)&0\\
0&G_j(K)
\end{array}
\right],
\end{equation*}

\item $\det G_j(K)=\det[CD-j^2I].$
\end{enumerate}
\end{lemma}

Using Lemma~\ref{det} one obtains the following.

\begin{lemma}{\label{biffunbldi}}
Let $H\in C^{2,0}(\R^{2n}\times\R^k,\R)$ and fix
$x_0\in\R^{2n}.$
Assume that condition~\eqref{bldi} is satisfied.
Define the functions $F_j\colon\R^k\rightarrow\R,$
$j\in\N\cup\set{0},$ by the formula
\begin{equation}{\label{Fj}}
F_j(\lambda)
=\det G_j(\nabla_x^2H(x_0,\lambda))
=\det[A(\lambda)B(\lambda)-j^2I].
\end{equation}
Then $\set{F_j}_{j\in\N\cup\set{0}}$ is a~detecting sequence
for~\eqref{parhamniel}.
\end{lemma}

Notice that the functions $F_j$ given by~\eqref{Fj} multiplied by
$\frac{1}{1+j^2}$ are equal to the functions $F_j$ from
Remark~\ref{biffun}.

Clearly, for every $j\in\N,$ $\lambda\in\R^k$ the function $F_j$
given by~\eqref{Fj} satisfies the condition
\begin{equation}{\label{maxrem}}
F_j(\lambda)=0 \quad \Leftrightarrow \quad
j^2\in\pspect{A(\lambda)B(\lambda)}
\quad \Leftrightarrow \quad
1\in\pspect{\frac{1}{j^2}A(\lambda)B(\lambda)}.
\end{equation}

\begin{lemma}{\label{maxgen}}
Let $H\in C^{2,0}(\R^{2n}\times\R^k,\R),$ fix $x_0\in\R^{2n},$ and
let $\set{F_j}_{j\in\N\cup\set{0}}\subset C(\R^k,\R)$ be
a~detecting sequence for~\eqref{parhamniel}. Then for every bounded
open set $U\subset\R^k$ the set
\begin{equation*}
\set{j\in\N\cup\set{0}\cond \exists_{\lambda\in U}:\;F_j(\lambda)=0}
\end{equation*}
is finite. Moreover, every
$\lambda_0\in\R^k$ has an open neighbourhood $U\subset\R^k$ such
that $F_j(\lambda)\neq 0$ for every $\lambda\in U$ and every
$j\in\N\cup\set{0}$ such that $F_j(\lambda_0)\neq 0.$
\end{lemma}
\begin{proof}
For every $j\in\N$ one has
$Q_j(\nabla_x^2H(x_0,\lambda))
=\frac{j}{j^2+1}\left(P+\frac{1}{j}Z(\lambda)\right),$
where
\begin{equation*}
P=\left[\begin{array}{cc}
0&J^t \\
J& 0
\end{array}\right],
\qquad Z(\lambda)=\left[\begin{array}{cc}
-\nabla_x^2H(x_0,\lambda)&0\\
0&-\nabla_x^2H(x_0,\lambda)
\end{array}\right].
\end{equation*}
Since $\sigma(P)=\set{-1,1},$ there exists $\varepsilon>0$
such that for every $T\in S(4n,\R)$ with the operator norm
$\norm{T}<\varepsilon$ one has
$\sigma(P+T)\cap(-\frac{1}{2},\frac{1}{2})=\emptyset,$
hence $\det(P+T)\neq 0.$ On the other hand,
$H\in C^{2,0}(\R^{2n}\times\R^k,\R),$ therefore for any bounded
open set $U\subset \R^k$ the number
$\displaystyle \sup_{\lambda\in U}\norm{Z(\lambda)}$
is finite. Thus for fixed $U$ there exists $m\in\N$ such that
$\frac{1}{j}\norm{Z(\lambda)}<\varepsilon$
for every $\lambda\in U$ and $j\in\N,$ $j>m.$ Consequently,
$F_j(\lambda)\neq 0$ for every $\lambda\in U,$ $j\in\N,$ $j>m.$
Now, choose $U$ to be a~neighbourhood of $\lambda_0.$ Since the set
$\set{0,\ldots,m}$ is finite and $F_0,\ldots,F_m$ are continuous,
one can change $U$ in such a~way that also $F_j(\lambda)\neq 0$ for
every $\lambda\in U$ and every $j\in\set{0,\ldots,m}$ such that
$F_j(\lambda_0)\neq 0.$
\end{proof}

Let $[a]$ denote the integer part of $a\in\R.$ One can use the
following lemma to find all the functions $F_j$ vanishing in
a~neighbourhood of given $\lambda_0\in\R^k$ in the case of systems
satisfying condition~\eqref{bldi}.

\begin{lemma}{\label{max}}
Let the assumptions of Lemma~\ref{biffunbldi} be satisfied.
Fix $\lambda_0\in\R^k$ and set
\begin{equation*}
N(\lambda_0)=
\left[\max\set{\sqrt{\nu}\cond
\nu\in\pspect{A(\lambda_0)B(\lambda_0)}}\right].
\end{equation*}
Then there exists an open neighbourhood $U\subset\R^k$ of
$\lambda_0$ such that $F_j(\lambda)\neq 0$ for every
$j>N(\lambda_0),$ $\lambda\in U.$
\end{lemma}
\begin{proof}
If $j>N(\lambda_0)$ then
$j^2>\max\pspect{A(\lambda_0)B(\lambda_0)},$ hence
$F_j(\lambda_0)\neq 0,$ in view of~\eqref{maxrem}.
Application of Lemma~\ref{maxgen} completes the proof.
\end{proof}

The following assumptions are used in the reminder of this paper.

\begin{trivlist}
\item[\textbf{(H1)}] $H\in C^{2,0}(\R^{2n}\times\R^k,\R),$

\item[\textbf{(H2)}] $x_0\in\R^{2n}$ and $\nabla_x
H(x_0,\lambda)=0$ for all $\lambda\in\R^k,$

\item[\textbf{(H3)}] $\set{F_j}_{j\in\N\cup\set{0}}
\subset C(\R^k,\R)$
is a~detecting sequence for~\eqref{parhamniel}.
\end{trivlist}

The set $\cT(x_0)=\set{x_0}\times\R^k$ is regarded as the set of
trivial solutions of~\eqref{parhamniel}.
In some theorems it is assumed additionally that for given
$\lambda\in\R^k$ the following conditions are satisfied.

\begin{trivlist}
\item[\textbf{(E1$(x_0,\lambda)$)}] There exists a~neighbourhood
$W\subset\R^{2n}\times\R^k$ of $(x_0,\lambda)$ such that
\begin{equation*}
(\nabla_xH)^{-1}(\set{0})\cap W\subset\set{x_0}\times\R^k
\end{equation*}
(i.e. $(x_0,\lambda)$ is not a~bifurcation point of nontrivial
stationary solutions of~\eqref{parhamniel}),
\end{trivlist}

\begin{trivlist}
\item[\textbf{(E2$(x_0,\lambda)$)}]
$\ind{\nabla_xH(\cdot,\lambda)}{x_0}\neq 0.$
\end{trivlist}

\begin{remark}{\label{inters}}
If conditions (H1)-(H3) are satisfied then
Lemma~\ref{isotropy} and Corollary~\ref{necniel} can be formulated
in terms of the functions $F_j,$ since for every
$j\in\N\cup\set{0},$ $\lambda_0\in\R^k$ one has
\begin{equation*}
\Lambda_j(\nabla_x^2H(x_0,\cdot))=F_j^{-1}(\set{0}),
\end{equation*}
\begin{equation*}
X(\lambda_0)=\set{j\in\N\cup\set{0}\cond F_j(\lambda_0)=0}.
\end{equation*}
\end{remark}

In what follows $\lambda_0\in\R^k$ is fixed.

Continuous curve in $\R^k$ means any subset of
$\R^k$ homeomorphic to $\R.$ A~submanifold of $\R^k$ is called
a~manifold and the tangent space to such a~manifold is regarded as
a~linear subspace of $\R^k.$

\begin{theorem}{\label{zdeg}}
Let conditions (H1)-(H3), (E1$(x_0,\lambda_0)$), and
(E2$(x_0,\lambda_0)$) be satisfied.
Assume that $M\subset\R^k$ is a~continuous curve and
$\lambda_0\in M$ is an isolated element of the set
\begin{equation*}
\bigcup_{l\in\N}F_{lj}^{-1}(\set{0})\cap M
\end{equation*}
for some $j\in\N.$ If the restriction of $F_{j}$ to $M$ changes its
sign at $\lambda_0$ then $(x_0,\lambda_0)$ is a~branching
point of nonstationary \mbox{$j$-solu}tions of~\eqref{parhamniel}
and a~global bifurcation point of nontrivial \mbox{$j$-solu}tions.
\end{theorem}
\begin{proof}
Let $\varphi\colon\R\rightarrow M$ be a~parametrization of $M$ such
that $\varphi(0)=\lambda_0$ and let
$H_1\colon\R^{2n}\times\R\rightarrow\R$ be the Hamiltonian defined
by $H_1(x,s)=H(x,\varphi(s)).$ It suffices
to prove the conclusion for $H_1$ and
$(x_0,0)\in\hilb\times\R$
instead of $H$ and $(x_0,\lambda_0)\in\hilb\times\R^k.$
By assumptions (E1$(x_0,\lambda_0)$), (E2$(x_0,\lambda_0)$) one has
$\ind{\nabla_x H_1(\cdot,\varepsilon)}{x_0}
=\ind{\nabla_x H_1(\cdot,-\varepsilon)}{x_0}
=\ind{\nabla_x H_1(\cdot,0)}{x_0}
=\ind{\nabla_xH(\cdot,\lambda_0)}{x_0}\neq 0,$
therefore
\begin{align*}
\eta_j(x_0,0) =\ind{\nabla_x H(\cdot,\lambda_0)}{x_0}
& \cdot\left(\frac{\morse{Q_j(\nabla_x^2H_1(x_0,\varepsilon))}}{2}
\right. \\
&\quad \left.
-\frac{\morse{Q_j(\nabla_x^2H_1(x_0,-\varepsilon))}}{2}\right).
\end{align*}
Since $F_{j}$ is a~\mbox{$j$th} detecting function and its
restriction to $M$ changes its sign at $\lambda_0,$ one has
\begin{align*}
\morse{Q_{j}(\nabla_x^2H_1(x_0,\varepsilon))} &
= \morse{Q_{j}(\nabla_x^2H(x_0,\varphi(\varepsilon))} \\
 & \neq \morse{Q_{j}(\nabla_x^2H(x_0,\varphi(-\varepsilon))} \\
& =\morse{Q_{j}(\nabla_x^2H_1(x_0,-\varepsilon)}.
\end{align*}
Thus $\eta_{j}(x_0,0)\neq 0,$ which implies that $(x_0,0)$ is
a~branching point of nonstationary \mbox{$j$-solu}\-tions and
a~global bifurcation point of nontrivial \mbox{$j$-solu}tions,
according to Theorem~\ref{genzdeg}.
\end{proof}

\begin{theorem}{\label{niezdeg}}
Let conditions (H1)-(H3) be satisfied. Assume that $M\subset\R^k$ is
a~continuous curve and $\lambda_0\in M$ is an isolated element of
the set
\begin{equation*}
\left(F_0^{-1}(\set{0})\cup\bigcup_{l\in\N}F_{lj}^{-1}(\set{0})\right)
\cap M
\end{equation*}
for some $j\in\N\cup\set{0}.$ If the restriction of $F_{j}$ to $M$
changes its sign at $\lambda_0$ then $(x_0,\lambda_0)$ is a~global
bifurcation point of nontrivial (possibly stationary)
\mbox{$j$-solu}tions of~\eqref{parhamniel}.
\end{theorem}
\begin{proof}
Choose the parametrization $\varphi$ and the modified Hamiltonian
$H_1$ as in the proof of Theorem~\ref{zdeg}. By the assumption
there exists $\varepsilon>0$ such that $F_0(\varphi(s))\neq 0$ for
$s\in[-\varepsilon,\varepsilon]\backslash\set{0}.$ If the
restriction of $F_0$ to $M$ changes its sign at $\lambda_0$ then
\begin{align*}
\eta_0(x_0,0) &
=\sgn\det\nabla_x^2H_1(x_0,\varepsilon)
-\sgn\det\nabla_x^2H_1(x_0,-\varepsilon) \\
& =\sgn F_0(\varphi(\varepsilon))
-\sgn F_0(\varphi(-\varepsilon))\neq 0,
\end{align*}
hence $(x_0,0)$ is a~global bifurcation point of nontrivial
stationary solutions, which are \mbox{$j$-solu}tions for every
$j\in\N.$ Thus one can assume that the restriction of $F_0$ to $M$
does not change its sign at $\lambda_0$ (in particular,
$j\neq 0$). Then one has
\begin{equation*}
\sgn\det\nabla_x^2H_1(x_0,\varepsilon)
=\sgn\det\nabla_x^2H_1(x_0,-\varepsilon)
=\sgn F(\varphi(\varepsilon))\neq 0,
\end{equation*}
and
\begin{equation*}
\eta_{j}(x_0,0) =\sgn F(\varphi(\varepsilon))
\cdot\frac{\morse{Q_{j}(\nabla_x^2H_1(x_0,\varepsilon))}
-\morse{Q_{j}(\nabla_x^2H_1(x_0,-\varepsilon))}}{2}.
\end{equation*}
Thus $\eta_{j}(x_0,0)\neq 0,$ similarly as in the proof of
Theorem~\ref{zdeg}.
\end{proof}

\begin{remark}{\label{signum}}
Let $k\geq 2,$ $\lambda_0\in\R^k,$ $r\in\N,$ $F\in C^r(\R^k,\R),$
$F(\lambda_0)=0,$ $\nabla F(\lambda_0)\neq 0.$ Then there exists
a~neighbourhood $U\subset\R^k$ of $\lambda_0,$ such that
$\Gamma=F^{-1}(\set{0})\cap U$ is a~\mbox{$(k-1)$-di}mensional
manifold of class $C^r.$ Note that if $L$ is a~one
dimensional linear subspace of $\R^k$ such that
$L\not\subset T_{\lambda_0}\Gamma$ then the restriction of $F$ to
the straight line $L_{\lambda_0}=\lambda_0+L$ has an isolated zero
at $\lambda_0$ and changes its sign at $\lambda_0.$
\end{remark}

Set
\begin{align*}
X_{j}(\lambda_0) &\eqdf\set{l\in\N\cup\set{0}\cond F_{lj}(\lambda_0)=0},
\quad j\in\N\cup\set{0}, \\
X_{j}^+(\lambda_0) &\eqdf\set{l\in\N\cond F_{lj}(\lambda_0)=0},
\quad j\in\N.
\end{align*}

In view of Remark~\ref{signum}, the next two theorems follow from
Theorems~\ref{zdeg} and~\ref{niezdeg} for
$M=L_{\lambda_0}=\lambda_0+L$, where
$L\not\subset T_{\lambda_0}(F_{lj}^{-1}(\set{0})\cap U)$
for all $l\in X_j^+(\lambda_0)$ and $l\in X_j(\lambda_0),$
respectively.

\begin{theorem}{\label{gradzdeg}}
Let conditions (H1)-(H3), (E1$(x_0,\lambda_0)$), and
(E2$(x_0,\lambda_0)$) be satisfied. Assume that
$F_{j}(\lambda_0)=0$ for some $j\in\N.$ If for all
$l\in X_{j}^+(\lambda_0)$ the functions $F_{lj}$ are of class $C^1$
in a~neighbourhood of $\lambda_0$ and
$\nabla F_{lj}(\lambda_0)\neq 0$
then $(x_0,\lambda_0)$ is a~branching point of
nonstationary \mbox{$j$-solu}tions of~\eqref{parhamniel}
and a~global bifurcation point of nontrivial \mbox{$j$-solu}tions.
\end{theorem}

\begin{theorem}{\label{gradniezdeg}}
Let conditions (H1)-(H3) be satisfied. Assume that
$F_{j}(\lambda_0)=0$ for some $j\in\N\cup\set{0}.$ If for all
$l\in X_{j}(\lambda_0)$ the functions $F_{lj}$ are of class $C^1$
in a~neighbourhood of $\lambda_0$ and
$\nabla F_{lj}(\lambda_0)\neq 0$
then $(x_0,\lambda_0)$ is a~global bifurcation point of nontrivial
(possibly stationary) \mbox{$j$-solu}tions of~\eqref{parhamniel}.
\end{theorem}

In view of Lemma~\ref{isotropy} and Remark~\ref{inters} one obtains
the following two pairs of corollaries to Theorems~\ref{gradzdeg}
and~\ref{gradniezdeg}, concerning symmetry breaking.
First consider the case of only one type of solutions in
a~neighbourhood of $(x_0,\lambda_0).$

\begin{corollary}{\label{wnzdegpure}}
Let conditions (H1)-(H3), (E1$(x_0,\lambda_0)$), and
(E2$(x_0,\lambda_0))$) be satisfied.
Fix $j\in\N.$ If $F_{j}$ is of class $C^1$ in a~neighbourhood of
$\lambda_0,$
$F_{j}(\lambda_0)=0,$ $\nabla F_{j}(\lambda_0)\neq 0,$ and
$F_{l}(\lambda_0)\neq 0$ for all $l\in\N,$ $l\neq j,$ then
$(x_0,\lambda_0)$ is a~global bifurcation point of nontrivial
\mbox{$j$-solu}tions of~\eqref{parhamniel}. Moreover, it is
a~branching point of nonstationary
solutions with the minimal period
$\frac{2\pi}{j},$ but it is not a~symmetry breaking point.
\end{corollary}

\begin{corollary}{\label{wnniezdpure}}
Let conditions (H1)-(H3) be satisfied. Fix $j\in\N\cup\set{0}.$ If
$F_{j}$ is of class $C^1$ in a~neighbourhood of  $\lambda_0,$
$F_{j}(\lambda_0)=0,$ $\nabla F_{j}(\lambda_0)\neq 0,$ and
$F_{l}(\lambda_0)\neq 0$ for all $l\in\N\cup\set{0},$ $l\neq j,$
then $(x_0,\lambda_0)$ is a~global bifurcation point of nontrivial
\mbox{$j$-solu}tions of~\eqref{parhamniel}. Moreover, it is
is a~branching point of
nonstationary solutions with the minimal
period $\frac{2\pi}{j}$ if $j\in\N,$ and nontrivial stationary
solutions if $j=0,$ but it is not a~symmetry breaking point.
\end{corollary}

The assumption of the next two corollaries, in which symmetry
breaking occurs, imply that $j_1$ and $j_2$ are relatively prime.

\begin{corollary}
Let conditions (H1)-(H3), (E1$(x_0,\lambda_0)$), and
(E2$(x_0,\lambda_0)$) be satisfied. Fix $j_1,j_2\in\N$ and assume
that for $i=1,2$ the functions $F_{j_i}$ are of class
$C^1$ in a~neighbourhood of $\lambda_0,$ $F_{j_i}(\lambda_0)=0,$
$\nabla F_{j_i}(\lambda_0)\neq 0,$ and $F_{lj_i}(\lambda_0)\neq 0$
for all $l\in\N,$ $l\geq 2.$ Then $(x_0,\lambda_0)$ is a~symmetry
breaking point. Namely, it is a~branching point of
nonstationary solutions of~\eqref{parhamniel} with the minimal
period $\frac{2\pi}{j_1}$ and solutions with the minimal period
$\frac{2\pi}{j_2}.$ Moreover, it is a~global bifurcation point
of nontrivial \mbox{$j_1$-solu}tions and \mbox{$j_2$-solu}tions.
\end{corollary}

\begin{corollary}
Let conditions (H1)-(H3) be satisfied. Fix $j_1,j_2\in\N$ and assume
that for $i=1,2$ the functions $F_{j_i}$ are of class $C^1$ in
a~neighbourhood of $\lambda_0,$ $F_{j_i}(\lambda_0)=0,$
$\nabla F_{j_i}(\lambda_0)\neq 0,$ and $F_{lj_i}(\lambda_0)\neq 0$
for all $l\in\N\cup\set{0},$ $l\neq 1.$ Then $(x_0,\lambda_0)$ is
a~symmetry breaking point. Namely, it is a~branching point of
nonstationary solutions of~\eqref{parhamniel}
with the minimal period $\frac{2\pi}{j_1}$ and solutions with
the minimal period $\frac{2\pi}{j_2}.$ Moreover, it is a~global
bifurcation point of nontrivial \mbox{$j_1$-solu}tions and
\mbox{$j_2$-solu}tions.
\end{corollary}

\section{The structure of the set of bifurcation
points}{\label{structsection}}

In this section the results from Section~\ref{gen}
and~\cite{Sf,SfMA} are applied to the description of the
structure of the set of bifurcation points of solutions
of~\eqref{parhamniel}.

Let $\Bif(x_0)$ and $\GlBif(x_0)$ be the sets of those
$\lambda\in\R^k$ for which $(x_0,\lambda)$ is, respectively,
a~bifurcation point and a~global bifurcation point of nontrivial
solutions of~\eqref{parhamniel}. Similarly, for every
$j\in\N\cup\set{0}$
let $\Bif_j(x_0)$ and $\GlBif_j(x_0)$ denote the sets of those
$\lambda\in\R^k$ for which $(x_0,\lambda)$ is, respectively,
a~bifurcation point and a~global bifurcation point of nontrivial
\mbox{$j$-so}lutions of~\eqref{parhamniel}. Finally, for every
$j\in\N\cup\set{0}$ let the subsets
$\Bif_j^{min}(x_0)\subset\Bif_j(x_0),$
$\GlBif_j^{min}(x_0)\subset\GlBif_j(x_0)$
consist of those $\lambda$ for which $(x_0,\lambda)$ is, respectively,
a~bifurcation point and a~branching point of nonstationary
solutions of~\eqref{parhamniel} with the minimal period
$\frac{2\pi}{j}$ if $j\in\N,$ and nontrivial stationary solutions
if $j=0.$

Let
\begin{align*}
X(\lambda)& \eqdf \set{j\in\N\cup\set{0}\cond F_j(\lambda)=0},\\
X^+(\lambda)& \eqdf \set{j\in\N\cond F_j(\lambda)=0}
=X(\lambda)\backslash\set{0},\\
X_{j}(\lambda)&
\eqdf \set{l\in\N\cup\set{0}\cond F_{lj}(\lambda)=0},
\quad j\in\N\cup\set{0},\\
X_{j}^+(\lambda)& \eqdf \set{l\in\N\cond F_{lj}(\lambda)=0}
=X_j(\lambda)\backslash\set{0}, \quad j\in\N.
\end{align*}
Set also
$\Psing(F)\eqdf F^{-1}(\set{0})\cap\nabla F^{-1}(\set{0}).$

As it is shown in the subsequent part of this paper,
Theorems~\ref{strzdegwh}-\ref{strniezd} bellow provide
a~constructive description of the set of bifurcation points of
solutions of~\eqref{parhamniel} which can be used both to obtain
qualitative results by applying theorems of real algebraic
geometry as well as in numerical computations for finding all
bifurcation points in given domain. Notice that the existence
of the neighbourhood $U$ of $\lambda_0$ is ensured by
Lemma~\ref{maxgen}.

\begin{theorem}{\label{strzdegwh}}
Let assumptions (H1)-(H3) be fulfilled and let $U\subset\R^k$ be
an~open neighbourhood of $\lambda_0\in\R^k$ such that the conditions
(E1$(x_0,\lambda)$), (E2$(x_0,\lambda)$), and
$F_m(\lambda)\neq 0$ are satisfied for every $\lambda\in U$ and
$m\in\N\backslash X^+(\lambda_0).$ If $X^+(\lambda_0)=\emptyset$
then $\Bif(x_0)\cap U=\emptyset.$ If
$X^+(\lambda_0)\neq\emptyset$ and $F_j,$ $j\in X^+(\lambda_0),$
are of class $C^1$ in $U$ then the following conclusions hold
for $\displaystyle F=\prod_{j\in X^+(\lambda_0)}F_{j}.$
\begin{enumerate}
\item{\label{bgenall}} $\displaystyle
      \Bif(x_0)\cap U\backslash\Psing(F)
      =\GlBif(x_0)\cap U\backslash\Psing(F) \\
      = F^{-1}(\set{0})\cap U \backslash\Psing(F)
      =\bigcup_{j\in X^+(\lambda_0)}F_{j}^{-1}(\set{0})
      \cap U\backslash\Psing(F).$
\item{\label{bgenfi}} For every $j\in X^+(\lambda_0)$ one has
      \begin{align*}
      \Bif_{j}^{min}(x_0)\cap U\backslash\Psing(F)
      & =\GlBif_{j}^{min}(x_0)\cap U\backslash\Psing(F) \\
      & =F_{j}^{-1}(\set{0})\cap U\backslash\Psing(F).
      \end{align*}
      The sets
      $\GlBif_{j}^{min}(x_0)\cap U\backslash\Psing(F),$
      $j\in X^+(\lambda_0),$ are pairwise disjoint.
\item{\label{bgenan}} If $F_j,$  $j\in X^+(\lambda_0),$
      are analytic in $U$ and $\overline{\lambda}$ is
      an isolated element of $\Psing(F)$ such that
      $\overline{\lambda}\in\cl{F_{j_0}^{-1}(\set{0})
      \backslash\Psing(F)}$ for some
      $j_0\in X^+(\lambda_0)$ then
      $\overline{\lambda}\in\Bif_{j_0}^{min}
      \cap \GlBif_{j_0}(x_0).$
\end{enumerate}
\end{theorem}
\begin{proof} If $X^+(\lambda_0)=\emptyset$
then $\Bif(x_0)\cap U=\emptyset,$ in view of Corollary~\ref{necniel}
and Remark~\ref{inters}. Assume that $X^+(\lambda_0)\neq\emptyset.$
Conclusion~\eqref{bgenall} follows from assertion~\eqref{bgenfi},
Corollary~\ref{necniel}, and Remark~\ref{inters}. To prove
assertion~\eqref{bgenfi} observe that
\begin{equation*}
\nabla F(\lambda)
=\sum_{j\in X^+(\lambda_0)}\nabla F_{j}(\lambda)
\prod_{i\in X^+(\lambda_0)\backslash\set{j}}F_{i}(\lambda).
\end{equation*}
Fix $j\in X^+(\lambda_0)$ and $\lambda\in U\backslash\Psing(F)$ such
that $F_j(\lambda)=0.$ Then $\nabla F_{j}(\lambda)\neq 0$ and
$F_{i}(\lambda)\neq 0$ for all
$i\in X^+(\lambda_0)\backslash\set{j}.$ (In particular, the sets
$F_{j}^{-1}(\set{0})\cap U\backslash\Psing(F),$
$j\in X^+(\lambda_0),$ are pairwise disjoint.)
Thus Corollary~\ref{wnzdegpure} with $\lambda_0$ replaced
by $\lambda$ implies that
$(x_0,\lambda)\in \GlBif_{j}^{min}(x_0)$ and $(x_0,\lambda)$
is not a symmetry breaking point.

Now turn to assertion~\eqref{bgenan}. Notice that
conclusion~\eqref{bgenfi} implies that
$(x_0,\overline{\lambda})$ is a~bifurcation point of solutions with
the minimal period $\frac{2\pi}{j_0}$ as a~cluster point of
such bifurcation points. It remains to show that
$\overline{\lambda}\in\GlBif_{j_0}(x_0).$ (One cannot use
Corollary~\ref{wnzdegpure}, since
$\nabla F(\overline{\lambda})=0$). In view
of the curve selection lemma for semianalytic sets there exists
a~continuous curve $M$ such that $\overline{\lambda}$ is an
isolated element of $F^{-1}(\set{0})\cap M$ and the restriction of
$F_{j_0}$ to $M$ changes its sign at $\overline{\lambda}.$
Consequently, according to Theorem~\ref{zdeg},
$(x_0,\overline{\lambda})$ is a~global bifurcation point of
\mbox{$j_0$-solu}tions.
\end{proof}

Applying Corollary~\ref{wnniezdpure} and Theorem~\ref{niezdeg}
instead of Corollary~\ref{wnzdegpure} and Theorem~\ref{zdeg} one
obtains the following theorem in which bifurcation of nontrivial
stationary solutions is allowed.

\begin{theorem}{\label{strniezdwh}}
Let assumptions (H1)-(H3) be fulfilled and let $U\subset\R^k$ be an
open neighbourhood of $\lambda_0\in\R^k$ such that
$F_m(\lambda)\neq 0$ for every $\lambda\in U$ and
$m\in\N\cup\set{0}\backslash X(\lambda_0).$
If $X(\lambda_0)=\emptyset$ then $\Bif(x_0)\cap U=\emptyset.$
If $X(\lambda_0)\neq\emptyset$ and $F_j,$ $j\in X(\lambda_0),$
are of class $C^1$ in $U$ then
conclusions~\eqref{bgenall}-\eqref{bgenan}
of Theorem~\ref{strzdegwh} hold true for
$\displaystyle F =\prod_{j\in X(\lambda_0)}F_{j}$
and $X^+(\lambda_0)$ replaced by $X(\lambda_0).$
\end{theorem}

If the functions $F$ in Theorems~\ref{strzdegwh}
and~\ref{strniezdwh} do not satisfy the assumptions of that
theorems, one can restrict the discussion to the set
of bifurcation points of \mbox{$j$-so}lutions for some fixed
$j,$ which leads to the following two theorems.

\begin{theorem}{\label{strzdeg}}
Let assumptions (H1)-(H3) be fulfilled and let $U\subset\R^k$ be an
open neighbourhood of $\lambda_0\in\R^k$ such that the conditions
(E1$(x_0,\lambda)$), (E2$(x_0,\lambda)$), and
$F_{mj}(\lambda)\neq 0$ are satisfied for some fixed $j\in \N$ and
all $\lambda\in U,$ $m\in \N\backslash X^+_j(\lambda_0).$ If
$X_{j}^+(\lambda_0)=\emptyset$ then $\Bif_j(x_0)\cap U=\emptyset.$
If $X_{j}^+(\lambda_0)\neq\emptyset$ and $F_{lj},$
$l\in X_{j}^+(\lambda_0),$ are of class $C^1$ in $U$ then the
following conclusions hold for
$\displaystyle F=\prod_{l\in X_{j}^+(\lambda_0)}F_{lj}.$
\begin{enumerate}
\item{\label{bselall}} $\displaystyle
      \Bif_j(x_0)\cap U\backslash\Psing(F)
      =\GlBif_j(x_0)\cap U\backslash\Psing(F) \\
      = F^{-1}(\set{0})\cap U\backslash\Psing(F)
      =\bigcup_{l\in X_j^+(\lambda_0)}
      F_{lj}^{-1}(\set{0})\cap U\backslash\Psing(F).$
\item{\label{bselfi}} For every $l\in X_{j}^+(\lambda_0)$
      one has
      \begin{align*}
      \Bif_{lj}^{min}(x_0)\cap U\backslash\Psing(F)
      & =\GlBif_{lj}^{min}(x_0)\cap U\backslash\Psing(F) \\
      & =F_{lj}^{-1}(\set{0})\cap U\backslash\Psing(F).
      \end{align*}
      The sets
      $\GlBif_{lj}^{min}(x_0)\cap U\backslash\Psing(F),$
      $l\in X_j^+(\lambda_0),$ are pairwise disjoint.
\item{\label{bselan}} If $F_{lj},$ $l\in X_j^+(\lambda_0),$
      are analytic in $U$ and $\overline{\lambda}$ is
      an isolated element of $\Psing(F)$ such that
      $\overline{\lambda}\in \cl{F_{l_0j}^{-1}(\set{0})\backslash\Psing(F)}$
      for some fixed $l_0\in X_j^+(\lambda_0)$ then
      $\overline{\lambda}\in\Bif_{l_0j}^{min}(x_0)
      \cap\GlBif_{l_0j}(x_0).$
\end{enumerate}
\end{theorem}

\begin{theorem}{\label{strniezd}}
Let assumptions (H1)-(H3) be fulfilled and let $U\subset\R^k$ be an
open neighbourhood of $\lambda_0\in\R^k$ such that
$F_{mj}(\lambda)\neq 0$ for some fixed $j\in\N\cup\set{0}$ and all
$\lambda\in U,$ $m\in\N\cup\set{0}\backslash X_j(\lambda_0).$ If
$X_j(\lambda_0)=\emptyset$ then $\Bif_j(x_0)\cap U=\emptyset.$
If $X_j(\lambda_0)\neq\emptyset$ and $F_{lj},$
$l\in X_{j}(\lambda_0),$ are of class $C^1$ in $U$ then
conclusions~\eqref{bselall}-\eqref{bselan} of
Theorem~\ref{strzdeg} hold true for
$\displaystyle F=\prod_{l\in X_{j}(\lambda_0)}F_{lj}$ and
$X_j^+(\lambda_0)$ replaced by $X_j(\lambda_0).$
\end{theorem}

\begin{remark}
In view of Lemma~\ref{maxgen},
Theorems~\ref{strzdegwh}-\ref{strniezd} remain valid for a~bounded
open set $U\subset\R^k$ and the sets $X^+(\lambda_0),$
$X(\lambda_0),$ $X_j^+(\lambda_0),$ $X_j(\lambda_0)$ replaced
by the sets
\begin{align*}
X^+(U) &\eqdf \set{j\in\N\cond
\exists_{\lambda\in U}:\;F_j(\lambda)=0},\\
X(U) &\eqdf \set{j\in\N\cup\set{0}\cond
\exists_{\lambda\in U}:\;F_j(\lambda)=0},\\
X_j^+(U) &\eqdf \set{l\in\N\cond
\exists_{\lambda\in U}:\;F_{lj}(\lambda)=0},\\
X_j(U) &\eqdf \set{l\in\N\cup\set{0}\cond
\exists_{\lambda\in U}:\;F_{lj}(\lambda)=0},
\end{align*}
respectively, which makes that theorems independent from
$\lambda_0.$
\end{remark}

Now the results from~\cite{Sf,SfMA}
can be applied. In what follows the symbols
$D^k_r,$ and $S^{k-1}_r$ denote, respectively,
the closed disk and the sphere in $\R^k$ centred at
the origin with radius $r>0.$

\begin{definition}{\label{admissible}}
A~mapping $F\colon\R^k\rightarrow\R$ is called \emph{admissible} if
it is analytic and $0\in\R^k$ is an isolated singular point of
$F^{-1}(\set{0}),$ i.e. it is an isolated element of the set
$F^{-1}(\set{0})\cap(\nabla F)^{-1}(\set{0}).$
\end{definition}

Consider first the case of two parameters ($k=2$).

\begin{definition}{\label{testfuncdef}}
Let $F\colon\R^2\to\R$ be admissible. An analytic mapping
$g\colon\R^2\rightarrow\R$ is called a~\emph{test function for $F$}
if $0\in\R^2$ is an isolated element of the set
$g^{-1}(\set{0})\cap F^{-1}(\set{0}).$
\end{definition}

Set $h(g,F)\eqdf (\Jac(g,F),F)\colon\R^2\rightarrow\R^2$
(see~\cite{Sf}), where $\Jac(g,F) \colon \R^2\rightarrow\R$
is the Jacobian of the mapping $(g,F)\colon\R^2\rightarrow\R^2.$

Applying Theorem~\ref{szafr}, Corollary~\ref{wnszafr}
(see Appendix~\ref{semianal}), and Lemma~\ref{maxgen} one obtains
the following two corollaries to Theorems~\ref{strzdegwh}
and~\ref{strniezdwh}. Determining the numbers
$b_+(g,F)$ i $b_-(g,F)$ in these corollaries allows to localize
the curves forming the set of bifurcation points (see for example
Corollary~\ref{wnIRcw}). Note that assumptions
(E1$(x_0,0)$) and (E2$(x_0,0)$) in Corollary~\ref{strzdegwhcor}
(and in Corollary~\ref{multzdeg}) imply that conditions
(E1$(x_0,\lambda)$) and (E2$(x_0,\lambda)$) are satisfied
for every $\lambda$ from a~neighbourhood of the origin.

\begin{corollary}{\label{strzdegwhcor}}
Let conditions (H1)-(H3), (E1$(x_0,0)$), and (E2$(x_0,0)$) be
satisfied for $k=2,$ and let $X^+(0)\neq\emptyset.$
Set $\displaystyle F=\prod_{j\in X^+(0)}F_{j}$ and assume that $F,$
$F_j,$ $j\in X^+(0),$ are admissible and $g_+$ is a~nonnegative test
function for $F.$ Then for every sufficiently small $r>0$ the
following conclusions hold.
\begin{enumerate}
\item{\label{cortwofir}} Each of the sets
      $\GlBif(x_0)\cap D^2_r\backslash\set{0},$
      $\GlBif_{j}^{min}(x_0)\cap D^2_r\backslash\set{0},$
      $j\in X^+(0),$
      is a~union of even (possibly zero) number of disjoint
      analytic curves, each of which
      meets the origin and crosses $S^1_r$ transversally in one
      point. The number of those curves, equal to $b(F)$ and
      $b(F_{j}),$ respectively, is determined
      by formula~\eqref{wnszafrform} in Corollary~\ref{wnszafr}.
      If the number of the curves is nonzero then
      $0\in \GlBif(x_0).$
\item{\label{cortwosec}} If $g$ is an arbitrary test
      function for $F$ then the
      number of those curves forming
      $\GlBif(x_0)\cap D^2_r\backslash\set{0}$ and
      $\GlBif_{j}^{min}(x_0)\cap D^2_r\backslash\set{0},$
      $j\in X^+(0),$ on which $g$ is positive (negative), equal
      to $b_+(g,F)$ ($b_-(g,F)$)
      and $b_+(g,F_{j})$ ($b_-(g,F_{j})$), respectively, is
      determined by formula~\eqref{szafrform} in
      Theorem~\ref{szafr}.
\item{\label{cortwothi}} If $b(F)\neq b(F_{j})\neq 0$
      for some $j\in X^+(0),$
      then $(x_0,0)$ is a~symmetry breaking point.
\end{enumerate}
\end{corollary}

\begin{corollary}{\label{strniezdwhcor}}
Let conditions(H1)-(H3) be satisfied for $k=2$ and let
$X(0)\neq\emptyset.$ Set $\displaystyle F=\prod_{j\in X(0)}F_{j}$
and assume that $F,$ $F_j,$ $j\in X(0),$ are admissible and $g_+$ is
a~nonnegative test function for $F.$ Then for every sufficiently
small $r>0$ conclusions~\eqref{cortwofir}-\eqref{cortwothi}
of Corollary~\ref{strzdegwhcor} hold true for
$X^+(0)$ replaced by $X(0).$
\end{corollary}

\begin{remark}{\label{remtwo}}
One obtains two analogous corollaries to Theorems~\ref{strzdeg}
and~\ref{strniezd} in the case of $k=2.$
\end{remark}

Now consider the case of arbitrary number $k$ of parameters.

Assume that an admissible function $F\colon\R^k\to\R$ is a~Morse
function on small spheres, i.e. there exists
$r>0$ such that $F|_{S^{k-1}_s}$ is a~Morse function for
every $0<s\leq r.$ Let $\Sigma_r$ be the set of
critical points of $F|_{S^{k-1}_r}.$ For $\lambda\in\Sigma_r$
denote by $\mathrm{ind}(F,\lambda)$ the Morse index of
$F|_{S^{k-1}_r}$ at $\lambda.$ Set
\begin{align*}
\nonumber n_+(F)&\eqdf\#\set{\lambda\in\Sigma_r\cond F(\lambda)<0
\;\wedge\;\mathrm{ind}(F,\lambda)\; \text{is even}}, \\
\nonumber n_-(F)&\eqdf\#\set{\lambda\in\Sigma_r\cond F(\lambda)<0
\;\wedge\;\mathrm{ind}(F,\lambda)\; \text{is odd}}, \\
\nonumber p_+(F)&\eqdf\#\set{\lambda\in\Sigma_r\cond F(\lambda)>0
\;\wedge\;\mathrm{ind}(F,\lambda)\; \text{is even}}, \\
\nonumber p_-(F)&\eqdf\#\set{\lambda\in\Sigma_r\cond F(\lambda)>0
\;\wedge\;\mathrm{ind}(F,\lambda)\; \text{is odd}}. \\
\end{align*}
Szafraniec~\cite{SfMA} proved theorems which can be used to verify
whether $F$ is Morse on small spheres and gave formulae for
$n_{\pm}(F),$ $p_{\pm}(F)$ written in terms of local topological
degree of mappings defined explicitly by using $F.$

Notice that if $n_{\mu}(F)\cdot p_{\nu}(F)\neq 0$ for some
$\mu,\nu\in\set{+,-}$ then $F$ has zeros on $S^{k-1}_r$ for every
sufficiently small $r>0.$ In particular,
$F^{-1}(\set{0})\cap D^k_r\neq \set{0}.$
Thus one obtains the following two
corollaries to Theorems~\ref{strzdegwh} and~\ref{strniezdwh}
(see also Lemma~\ref{maxgen}).

\begin{corollary}{\label{multzdeg}}
Let conditions (H1)-(H3), (E1$(x_0,0)$), and (E2$(x_0,0)$) be
satisfied, and let $X^+(0)\neq\emptyset.$
Set $\displaystyle F=\prod_{j\in X^+(0)}F_{j}$ and assume that $F,$
$F_j,$ $j\in X^+(0),$ are admissible and Morse on small spheres.
Then for every sufficiently small $r>0$ the following conclusions
hold.
\begin{enumerate}
\item{\label{corarbfir}} If $n_{\mu}(F)\cdot p_{\nu}(F)\neq 0$
      for some $\mu,\nu\in\set{+,-}$
      then the set $\GlBif(x_0)\cap D^k_r$
      is a~topological cone with vertex at the origin and base
      $F^{-1}(\set{0})\cap S^k_r.$
      Moreover, $\GlBif(x_0)\cap D^k_r\backslash\set{0}$ is
      a~\mbox{$(k-1)$-di}mensional manifold with boundary
      $F^{-1}(\set{0})\cap S^k_r.$

\item{\label{corarbsec}} Similarly, if
      $n_{\mu}(F_j)\cdot p_{\nu}(F_j)\neq 0$
      for some $j\in X^+(0)$ and $\mu,\nu\in\set{+,-}$
      then the set $\GlBif_j^{min}(x_0)\cap D^k_r\cup\set{0}$
      is a~topological cone with vertex at the origin and base
      $F_j^{-1}(\set{0})\cap S^k_r.$
      Moreover, $\GlBif_j^{min}(x_0)\cap D^k_r\backslash\set{0}$
      is a~\mbox{$(k-1)$-di}mensional manifold with boundary
      $F_j^{-1}(\set{0})\cap S^k_r.$

\item{\label{corarbthir}} If
      $n_{\mu_1}(F_{j_1})\cdot p_{\nu_1}(F_{j_1})
      \cdot n_{\mu_2}(F_{j_2})\cdot p_{\nu_2}(F_{j_2})\neq 0$
      for some $j_1,j_2\in X^+(0)$ and some
      $\mu_1,\nu_1,\mu_2,\nu_2\in\set{+,-}$
      then $(x_0,0)$ is a~symmetry breaking point.
\end{enumerate}
\end{corollary}

\begin{corollary}{\label{multstat}}
Let conditions(H1)-(H3) be satisfied and let $X(0)\neq\emptyset.$
Set $\displaystyle F=\prod_{j\in X(0)}F_{j}$ and assume that $F,$
$F_j,$ $j\in X(0),$ are admissible and Morse on small spheres.
Then for every sufficiently small $r>0$
conclusions~\eqref{corarbfir}-\eqref{corarbthir}
of Corollary~\ref{multzdeg} hold true for $X^+(0)$
replaced by $X(0).$
\end{corollary}

One obtains two analogous corollaries to Theorems~\ref{strzdeg}
and~\ref{strniezd}.

The above corollaries allow to use the formulae for $n_{\pm}(F),$
$p_{\pm}(F)$ given in~\cite{SfMA} to detect symmetry breaking
points. The results from~\cite{SfMA} can be also used to
investigate the number of the cones from the above corollaries.

\section{Examples}{\label{examples}}

In this section the results from Section~\ref{structsection} are
applied to examples of system~\eqref{parhamniel} with two and three
parameters. Symbolic computations of topological indices have been
performed by using {\L}\c{e}cki's program based on an algorithm
described in~\cite{EL,LS}. Other symbolic computations (solving
polynomial equations, estimates, etc.) have been carried out
by using Maple. The graphs of curves and surfaces forming the sets
of zeros of detecting functions (which are proved to be bifurcation
points of given systems) in prescribed area have been obtained by
using Endrass' program \emph{surf}~\cite{surf}.

Recall that $D^k_r$ denotes the closed disc in $\R^k$ centred at
the origin with radius $r>0.$

\begin{remark}{\label{singtang}}
Let $F\colon\R^k\to\R$ be an analytic function for some $k\in\N.$
Fix $r>0$ and let $0\in\R^k$ be the unique singular point
of $F$ in $D^k_r.$ Assume that for every
$\lambda\in F^{-1}(\set{0})\cap D^k_r\backslash\set{0}$ the tangent
space to $F^{-1}(\set{0})$ at $\lambda$ is not equal to the tangent
space at $\lambda$ to the sphere centred at the origin.
Then every connected component of $F^{-1}(\set{0})\cap D^k_r$
contains the origin.
\end{remark}

\begin{remark}
Let $U\subset\R^k$ be a~bounded open neighbourhood of
$\lambda_0\in\R^k$ and let $F\colon\R^k\to\R$ be a~continuous
function. In the aim of proving that $\lambda_0$ is the only zero
of $F$ in $\cl{U}$ it suffices to prove that $\lambda_0$ is the
only solution in $\R^k$ of the equation
\begin{equation*}
F(\lambda)^2=h(\lambda),
\end{equation*}
where $h\colon\R^k\to\R$ is a~continuous function such that
$h(\lambda_0)=0,$ $h(\lambda)=0$ for every $\lambda\in\partial U,$
$h(\lambda)>0$ for every
$\lambda\in\Inter{U}\backslash\set{\lambda_0},$
and $h(\lambda)<0$ for every $\lambda\in\R^k\backslash\cl{U}.$
If $\cl{U}$ is a~closed disc centred at $\lambda_0$ with radius
$r>0$ then one can exploit the function $h$ of the form
\begin{equation*}
h(\lambda)=
-\abs{\lambda-\lambda_0}^{2p}(\abs{\lambda-\lambda_0}^{2q}-r^{2q}),
\end{equation*}
where $p,q\in\N.$
\end{remark}

\begin{remark}{\label{probnerem}}
In the next two examples, the functions
$g_i\colon\R^2\rightarrow\R,$ $i=1,\ldots,4,$ are defined by
\begin{equation*}
g_1(\lambda_1,\lambda_2)=\lambda_1^2+\lambda_2^2,\;\;\;
g_2(\lambda_1,\lambda_2)=\lambda_1,\;\;\;
g_3(\lambda_1,\lambda_2)=\lambda_2,\;\;\;
g_4(\lambda_1,\lambda_2)=\lambda_1\cdot\lambda_2
\end{equation*}
(see~\cite{IR}). If $F\colon\R^2\rightarrow\R$ is
an admissible mapping which has no zeros on the coordinate
axes in a neighbourhood of the origin
(e.g. if $F(\cdot,0)$ and $F(0,\cdot)$ are polynomials
of nonzero degree) then $g_i,$ $i=1,\ldots,4,$ are test
functions for $F,$ since the roots of $g_i$ lie on
the coordinate axes. For such an $F$
the symbol $b_i(F)$ denotes the number of components of
$F^{-1}(\set{0})\cap D^2_r\backslash\set{0}$
(for sufficiently small $r>0$) contained in the
\mbox{$i$th} quarter of the plane $\R^2$ for $i=1,\ldots,4.$
\end{remark}

In all examples use is made of the functions $F_j$ defined
by~\eqref{Fj}.

\begin{example}{\label{exdeg}}
Let $H\colon\R^6\times\R^2\rightarrow\R$ be the Hamilton function
given by
\begin{equation*}
H(x,\lambda)\equiv
H(x_1,\ldots,x_6,\lambda_1,\lambda_2)
=P(x_1,\ldots,x_6,\lambda_1,\lambda_2)
+Q(x_1,\ldots,x_6),
\end{equation*}
where
\begin{align*}
P(x_1,\ldots,x_6,\lambda_1,\lambda_2)
&= \frac{1}{2}(9+\frac{1}{10}\lambda_1^6)x_1^2
+\frac{5}{2}x_3^2+\frac{1}{2}x_4^2
+\frac{1}{2}(5-\frac{1}{20}\lambda_1^5)x_6^2 \\
&\quad +2\lambda_2^6x_1x_3
+8\lambda_2^6x_4x_6+x_2^4+x_5^4,
\end{align*}
$Q\in C^2(\R^6,\R)$ has a~local minimum at the origin,
and $\nabla^2Q(0)=0,$ for example
\begin{align}{\label{exdegadd}}
\nonumber Q(x_1,\ldots,x_6) &
=x_1^6+x_2^6+x_3^6+x_4^6+x_5^6+x_6^6+(x_1^9+x_3^7)x_2 \\
& \quad +x_1x_2^7+x_2^7x_3+x_2x_6^9+x_4^7x_5+(x_2^9+x_5^7)x_6.
\end{align}

$H$ satisfies conditions (H1)-(H3) for $k=2$ and
$x_0=0\in\R^6.$

First the set of bifurcation points in  $\set{0}\times D^2_r$ will
be described for sufficiently small $r>0$ and then it will be shown
that the conclusions hold for every $r\leq 0.3.$

One has
\begin{equation*}
A(\lambda_1,\lambda_2) = \left[
\begin{array}{ccc}
9+\frac{1}{10}\lambda_1^6 & 0 & 2\lambda_2^6 \\
0 & 0 & \\
2\lambda_2^6 & 0 & 5
\end{array}
\right],
\end{equation*}
\begin{equation*}
B(\lambda_1,\lambda_2) = \left[
\begin{array}{ccc}
1 & 0 & 8\lambda_2^6 \\
0 & 0 & 0 \\
8\lambda_2^6 & 0 & 5-\frac{1}{20}\lambda_1^5
\end{array}
\right].
\end{equation*}
Those of the functions $F_j,$ $j\in\N\cup\set{0},$ defined
by~\eqref{Fj}, which vanish at $(0,0)\in\R^2$ are $F_0,$ $F_3,$
and $F_5,$ hence $X^+(0)=\set{3,5}$ (see also Lemmas~\ref{maxgen},
\ref{max}). One has
\begin{align*}
F_0(\lambda_1,\lambda_2) & \equiv  0, \\
F_3(\lambda_1,\lambda_2) &
=288\lambda_1^6\lambda_2^{12}
-\frac{72}{5}\lambda_1^6+\frac{9}{40}\lambda_1^{11}
-2304\lambda_2^{24}+
28692\lambda_2^{12}-\frac{9}{5}\lambda_1^5\lambda_2^{12},\\
F_5(\lambda_1,\lambda_2) &
=800\lambda_1^6\lambda_2^{12}+\frac{5}{8}\lambda_1^{11}
+92500\lambda_2^{12}
-100\lambda_1^5-6400\lambda_2^{24}-5\lambda_1^5\lambda_2^{12}.
\end{align*}

Use will be made of Theorems~\ref{strzdegwh}, \ref{strzdeg},
and Corollary~\ref{strzdegwhcor} (see also
Remark~\ref{remtwo}). Theorems~\ref{strniezdwh}, \ref{strniezd},
and Corollary~\ref{strniezdwhcor} cannot be applied,
since $F_0=0,$ which means that all the points
$(x_0,\lambda),$ $\lambda\in\R^2,$ are degenerate.

Observe that conditions (E1$(0,\lambda)$) and (E2$(0,\lambda)$) are
satisfied for every $\lambda\in U \eqdf (-0.31,0.31)^2$.
(At the moment only assumptions (E1$(0,0)$) and (E2$(0,0)$) are
needed, as in Corollary~\ref{strzdegwhcor}.)
Indeed. An appropriate estimate for the function $P$ shows that for
every $\lambda\in (-0.31,0.31)^2$ and every
$v\in\R^6\backslash\set{0}$ the function
$[0,+\infty)\ni c\mapsto P(cv,\lambda)$ is strictly increasing
(in particular, $P(\cdot,\lambda)$ has a~strict local minimum at
$0\in\R^6$). On the other hand, $Q$ has a~minimum at $0\in\R^6$ and
it does not depend on $\lambda.$ Thus there exists $\varepsilon>0$
such that $\nabla_xH(x,\lambda)\neq 0$ for every
$0<\abs{x}<\varepsilon,$ $\lambda\in (-0.31,0.31)^2.$
Consequently, for every $\lambda\in (-0.31,0.31)^2$ condition
(E1$(0,\lambda)$) is fulfilled and the function
$H(\cdot,\lambda)$ has a~strict local minimum at $0\in\R^6,$ hence
$\ind{\nabla_xH(\cdot,\lambda)}{0}=1\neq 0$ (see~\cite{A}).

Symbolic computations show that $(0,0)\in\R^2$ is an isolated
singular point of the functions $F_3,$ $F_5,$ and $F=F_3\cdot F_5.$
Thus they are admissible and, according to Remark~\ref{probnerem},
$g_i,$ $i=1,\ldots,4,$ are test
functions for them. Furthermore,
\begin{equation}{\label{indexfirgen}}
\ind{h(g_1,F_3)}{0}=2,\;\;
\ind{h(g_1,F_5)}{0}=1,\;\;\ind{h(g_1,F)}{0}=3,
\end{equation}
which has been checked by using {\L}\c{e}cki's program. It
follows from Theorems~\ref{strzdegwh},~\ref{strzdeg}
and Corollary~\ref{strzdegwhcor} (for $g_+=g_1$) that for
every sufficiently small $r>0$ the following equalities hold.
\begin{equation}{\label{bifzerfir}}
\begin{aligned}
\Bif(0)\cap D^2_r=\GlBif(0)\cap D^2_r
& =F^{-1}(\set{0})\cap D^2_r \\
& =\left(\GlBif_3^{min}(0)\cup
\GlBif_5^{min}(0)\right)\cap D^2_r, \\
\GlBif_3^{min}\cap D^2_r & =F_3^{-1}(\set{0})\cap D^2_r, \\
\GlBif_5^{min}\cap D^2_r & =F_5^{-1}(\set{0})\cap D^2_r.
\end{aligned}
\end{equation}
The fact that $0\in \GlBif_3^{min}\cap D^2_r$ and
$0\in \GlBif_5^{min}\cap D^2_r$ follows from Lemma~\ref{isotropy}
and Remark~\ref{inters}. (The only minimal periods of nontrivial
solutions in a~neighbourhood of the origin are $\frac{2\pi}{3}$ and
$\frac{2\pi}{5}.$)

In view of Corollary~\ref{strzdegwhcor}, the set
$\GlBif(0)\cap D^2_r\backslash\set{0}$ consists
of $b(F)=6$ curves. The numbers of the curves forming the sets
$\GlBif_3^{min}(0)\cap D^2_r\backslash\set{0}$ and
$\GlBif_5^{min}(0)\cap D^2_r\backslash\set{0}$ are equal to
$b(F_3)=4$ and $b(F_5)=5,$ respectively.

To localize the curves in the quarters of the plane
Corollary~\ref{wnIRcw} will be used. Application of {\L}\c{e}cki's
program yields
\begin{equation}{\label{indexfirspec}}
\begin{aligned}
\ind{h(g_2,F_3)}{0} & =  0,\;\;\ind{h(g_2,F_5)}{0}=1, \\
\ind{h(g_3,F_3)}{0} & =  0,\;\;\ind{h(g_3,F_5)}{0}=0, \\
\ind{h(g_4,F_3)}{0} & =  0,\;\;\ind{h(g_4,F_5)}{0}=0.
\end{aligned}
\end{equation}
Taking into account~\eqref{indexfirgen}, \eqref{indexfirspec},
and Corollary~\ref{wnIRcw} one obtains
\begin{align*}
b_1(F_3) & =  1,\;\; b_2(F_3)=1,\;\; b_3(F_3)=1,\;\;b_4(F_3)=1, \\
b_1(F_5) & =  1,\;\; b_2(F_5)=0,\;\; b_3(F_5)=0,\;\;b_4(F_5)=1.
\end{align*}

\begin{figure}
\includegraphics{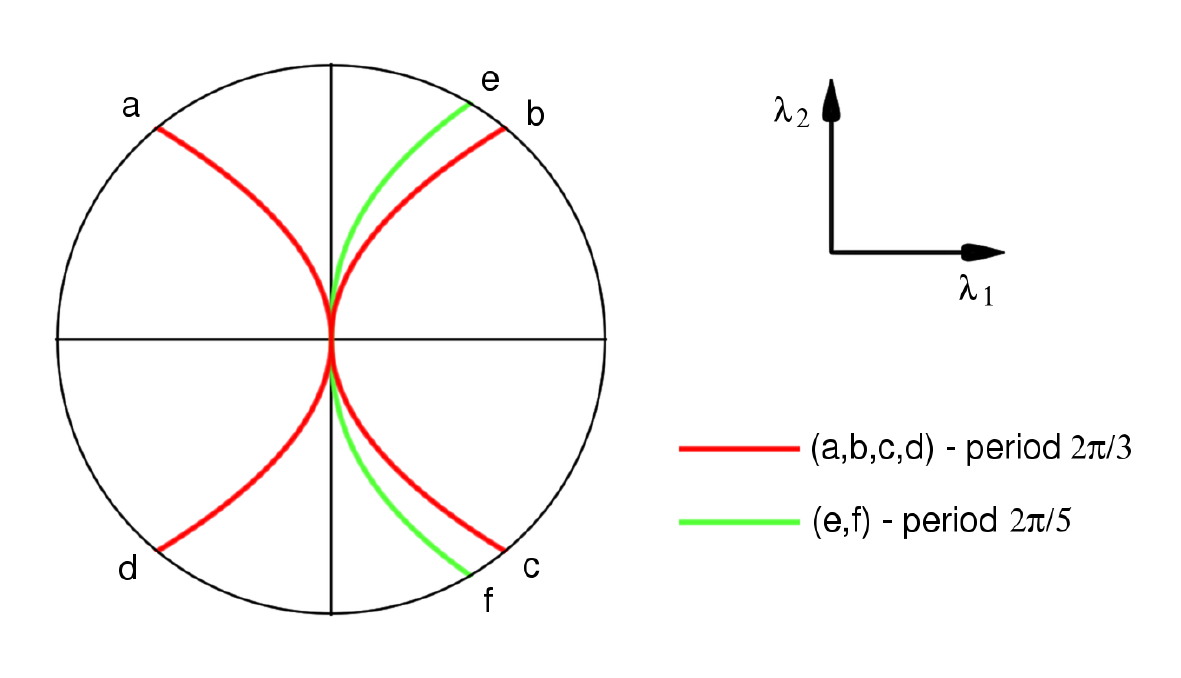}
\caption{The set of those $(\lambda_1,\lambda_2)\in D^2_r,$
$r=0.3,$ for which $(0,(\lambda_1,\lambda_2))\in\R^6\times\R^2$ is
a~global bifurcation point of the system from Example~\ref{exdeg}.
The legend on the right describes the minimal periods of solutions
bifurcating from the points of given curve.}
\label{figure1}
\end{figure}

The following results of additional symbolic computations and
estimates ensure that the above conclusions concerning bifurcation
points in $\set{0}\times D^2_r$ hold for every $r\leq 0.3.$ One has
$F_j(\lambda)\neq 0$ for every $j\in\N\backslash\set{3,5},$
$\lambda\in U\eqdf (-0.31,0.31)^2.$ The origin is the only singular
point of $F_3,$ $F_5,$ and $F=F_3\cdot F_5$ in $U.$ The sets of
zeros of $F_3,$ $F_5$ restricted to $D^2_r\backslash\set{0},$
$r=0.3,$ are disjoint and they have no common points with
the coordinate axes. Furthermore, the functions $F_3,$ $F_5,$
and $F$ satisfy the assumptions of Remark~\ref{singtang}
for $k=2$ and $r=0.3.$ Thus for $r=0.3$ every connected component
of $F_3^{-1}(\set{0})\cap D^2_r,$ $F_5^{-1}(\set{0})\cap D^2_r,$
and $F^{-1}(\set{0})\cap D^2_r$ contains the origin.

Theorems~\ref{strzdegwh} and~\ref{strzdeg} have been also applied
to find bifurcation points in $\set{0}\times D^2_r,$ $r=0.3,$
numerically as zeros of the functions $F_j,$ according to
formulae~\eqref{bifzerfir}, which has been performed by using the
program \emph{surf} and presented on Figure~\ref{figure1}.
The earlier conclusions ensure that the number of curves on
Figure~\ref{figure1}, their localization, and their relative
position do not change when passing to a~smaller scale.

One can summarize the above results as follows. The set of
bifurcation points in $\set{0}\times D^2_r,$ $r=0.3,$ is equal
to the set of global bifurcation points in this domain and consists
of six curves, for which the origin is the only common point. Apart
from the origin four curves (one curve in each quarter) consist of
branching points of solutions with the minimal period
$\frac{2}{3}\pi$ (and only such solutions), whereas two curves
(one curve in the first quarter and one curve in the fourth
quarter) consist of branching points of solutions with
the minimal period $\frac{2}{5}\pi$ (and only such solutions).
The origin is a~branching point of solutions with the
minimal periods $\frac{2}{3}\pi$ and solutions with the minimal
period $\frac{2}{5}\pi$ (and only such solutions). In particular,
the origin is a~symmetry breaking point.
\end{example}

\begin{example}{\label{exstat}}
Consider the Hamiltonian $H\colon\R^6\times\R^2\rightarrow\R$
defined by the formula
\begin{equation*}
H(x,\lambda)\equiv
H(x_1,\ldots,x_6,\lambda_1,\lambda_2)
=P(x_1,\ldots,x_6,\lambda_1,\lambda_2)
+Q(x_1,\ldots,x_6),
\end{equation*}
where
\begin{align*}
& P(x_1,\ldots,x_6,\lambda_1,\lambda_2) \\
& =\frac{1}{2}(4+3\lambda_1^{10}+\lambda_1^7\lambda_2
-\lambda_1^2\lambda_2^7
+\lambda_1^5\lambda_2^5-\lambda_1^4\lambda_2^5
-\lambda_1^3\lambda_2^4-\lambda_2^9)x_1^2 \\
&\quad +\frac{3}{2}x_2^2+x_3^2+\frac{1}{2}x_4^2
+\frac{1}{2}(3+3\lambda_1^7)x_5^2 \\
&\quad +\frac{1}{2}\lambda_1^2x_6^2
+(3\lambda_1^3+22\lambda_2^4)x_5x_6,\\
\end{align*}
\begin{equation}{\label{exstatadd}}
Q(x_1,\ldots,x_6)=x_1^3x_2+(x_1+x_4)x_3^2+x_4^2x_5+(x_6-x_5)^3.
\end{equation}

$H$ satisfies conditions (H1)-(H3) for $k=2$ and $x_0=0\in\R^6.$

Notice that in this case $x_0=0\in\R^6$ is an isolated critical
point of $H(\cdot,0),$ it is degenerate, and
$\ind{\nabla_xH(\cdot,0)}{0}=0.$

First the set of bifurcation points in  $\set{0}\times D^2_r$ will
be described for sufficiently small $r>0$ and then it will be
shown that the conclusions hold for every $r\leq 0.3.$

One has
\begin{equation*}
A(\lambda_1,\lambda_2) =\left[
\begin{array}{ccc}
4+3\lambda_1^{10}+\lambda_1^7\lambda_2-\lambda_1^2\lambda_2^7
+\lambda_1^5\lambda_2^5-\lambda_1^4\lambda_2^5
-\lambda_1^3\lambda_2^4-\lambda_2^9
& 0 & 0 \\
0 & 3 & 0 \\
0 & 0 & 2
\end{array}
\right],
\end{equation*}
\begin{equation*}
B(\lambda_1,\lambda_2) =\left[
\begin{array}{ccc}
1 & 0 & 0 \\
0 & 3+3\lambda_1^7 & 3\lambda_1^3+22\lambda_2^4 \\
0 & 3\lambda_1^3+22\lambda_2^4 & \lambda_1^2
\end{array}
\right].
\end{equation*}

Use will be made of Theorems~\ref{strniezdwh}, \ref{strniezd},
and Corollary~\ref{strniezdwhcor}
(see also Remark~\ref{remtwo}). Theorems~\ref{strzdegwh},
\ref{strzdeg}, and Corollary~\ref{strzdegwhcor}
are not suitable in this case. (It will be shown that the origin
is a~bifurcation point of nontrivial stationary solutions.)

Those of the functions $F_j,$ $j\in\N\cup\set{0},$ defined
by~\eqref{Fj}, which vanish at $(0,0)\in\R^2$ are
$F_0,$ $F_2,$ and $F_3,$ hence $X(0)=\set{0,2,3}$
(see also Lemmas~\ref{maxgen}, \ref{max}). One has
\begin{equation*}
F_0=f_0\cdot a_0,\;\;\;F_2=f_2\cdot a_2,\;\;\;F_3=f_3\cdot a_3,
\end{equation*}
where
\begin{align*}
f_0(\lambda_1,\lambda_2)
& =18\lambda_1^9+18\lambda_1^2-54\lambda_1^6
-792\lambda_1^3\lambda_2^4-2904\lambda_2^8,\\
f_2(\lambda_1,\lambda_2)
&=3\lambda_1^{10}+\lambda_1^7\lambda_2-\lambda_1^2\lambda_2^7
+\lambda_1^5\lambda_2^5-\lambda_1^4\lambda_2^5
-\lambda_1^3\lambda_2^4-\lambda_2^9, \\
f_3(\lambda_1,\lambda_2)
&=18\lambda_1^9-81\lambda_1^7-54\lambda_1^6
-792\lambda_1^3\lambda_2^4-2904\lambda_2^8,
\end{align*}
\begin{align*}
a_0(\lambda_1,\lambda_2) &=f_2(\lambda_1,\lambda_2)+4, \\
a_2(\lambda_1,\lambda_2)
& =f_0(\lambda_1,\lambda_2)-8\lambda_1^2-36\lambda_1^7-20, \\
a_3(\lambda_1,\lambda_2) &=a_0(\lambda_1,\lambda_2)-9.
\end{align*}
The functions $a_0,$ $a_2,$ $a_3$ have no zeros in
$U\eqdf (-0.31,0.31)^2.$
Thus $F_0,$ $F_2,$ $F_3$ can be replaced by $f_0,$ $f_2,$ $f_3$ in
computations.

It has been checked by symbolic computations that
$(0,0)$ is an isolated singular point of the
functions $F_0,$ $F_2,$ $F_3,$ and $F=F_0\cdot F_2\cdot F_3.$ Thus
they are admissible and, according to Remark~\ref{probnerem}, $g_i,$
$i=1,\ldots,4,$ are test functions for them.
Application of {\L}\c{e}cki's program gives
\begin{equation}{\label{indexsecgen}}
\begin{aligned}
& \ind{h(g_1,F_0)}{0}=2,\quad \ind{h(g_1,F_2)}{0}=3, \\
& \ind{h(g_1,F_3)}{0}=2,\quad \ind{h(g_1,F)}{0}=7.
\end{aligned}
\end{equation}
In view of Theorems~\ref{strniezdwh},~\ref{strniezd}
and Corollary~\ref{strniezdwhcor} (for $g_+=g_1$), for every
sufficiently small $r>0$ one has
\begin{equation}{\label{glbifexsecfir}}
\begin{aligned}
\Bif(0)\cap D^2_r &=\GlBif(0)\cap D^2_r=F^{-1}(\set{0})\cap D^2_r \\
& =\left(\GlBif_0(0)\cup \GlBif_2(0)\cup \GlBif_3(0)\right)
\cap D^2_r,
\end{aligned}
\end{equation}

\begin{equation}{\label{glbifexsecbis}}
\begin{aligned}
\GlBif_0^{min}(0)\cap D^2_r \equiv \GlBif_0(0)\cap D^2_r
& =F_0^{-1}(\set{0})\cap D^2_r, \\
\GlBif_2(0)\cap D^2_r
& =\left(F_0^{-1}(\set{0})\cup F_2^{-1}(\set{0})\right)
\cap D^2_r, \\
\GlBif_3(0)\cap D^2_r
& =\left( F_0^{-1}(\set{0})\cup F_3^{-1}(\set{0})\right)
\cap D^2_r,
\end{aligned}
\end{equation}

\begin{equation}{\label{glbifexsecthir}}
\begin{aligned}
\GlBif_2^{min}(0)\cap D^2_r\backslash\set{0}
& =  F_2^{-1}(\set{0})\cap D^2_r\backslash\set{0}, \\
\GlBif_3^{min}(0)\cap D^2_r\backslash\set{0}
& =F_3^{-1}(\set{0})\cap D^2_r\backslash\set{0}.
\end{aligned}
\end{equation}

According to Corollary~\ref{strniezdwhcor}, the set
$\GlBif(0)\cap D^2_r\backslash\set{0}$ consists of $b(F)=14$
curves. The numbers of curves forming the sets
$\GlBif_0^{min}(0)\cap D^2_r\backslash\set{0},$
$\GlBif_2^{min}(0)\cap D^2_r\backslash\set{0},$
$\GlBif_3^{min}(0)\cap D^2_r\backslash\set{0}$
are equal to $b(F_0)=4,$ $b(F_2)=6,$
$b(F_3)=4,$ respectively.

In the aim of applying Corollary~\ref{wnIRcw} to localize
the curves in the quarters of the plane it has been checked
by using {\L}\c{e}cki's program that
\begin{equation}{\label{indexsecspec}}
\begin{aligned}
\ind{h(g_2,F_0)}{0} & =  0,\;\;\ind{h(g_2,F_2)}{0}=0,\;\;
\ind{h(g_2,F_3)}{0}=-2, \\
\ind{h(g_3,F_0)}{0} & =  0,\;\;\ind{h(g_3,F_2)}{0}=1,\;\;
\ind{h(g_3,F_3)}{0}=0, \\
\ind{h(g_4,F_0)}{0} & =  0,\;\;\ind{h(g_4,F_2)}{0}=-2,\;\;
\ind{h(g_4,F_3)}{0}=0.
\end{aligned}
\end{equation}
Taking into account~\eqref{indexsecgen}, \eqref{indexsecspec},
and Corollary~\ref{wnIRcw} one obtains
\begin{align*}
b_1(F_0) & =  1,\;\; b_2(F_0)=1,\;\; b_3(F_0)=1,\;\;b_4(F_0)=1, \\
b_1(F_2) & =  1,\;\; b_2(F_2)=3,\;\; b_3(F_2)=0,\;\;b_4(F_2)=2, \\
b_1(F_3) & =  0,\;\; b_2(F_3)=2,\;\; b_3(F_3)=2,\;\;b_4(F_3)=0.
\end{align*}

\begin{figure}
\includegraphics{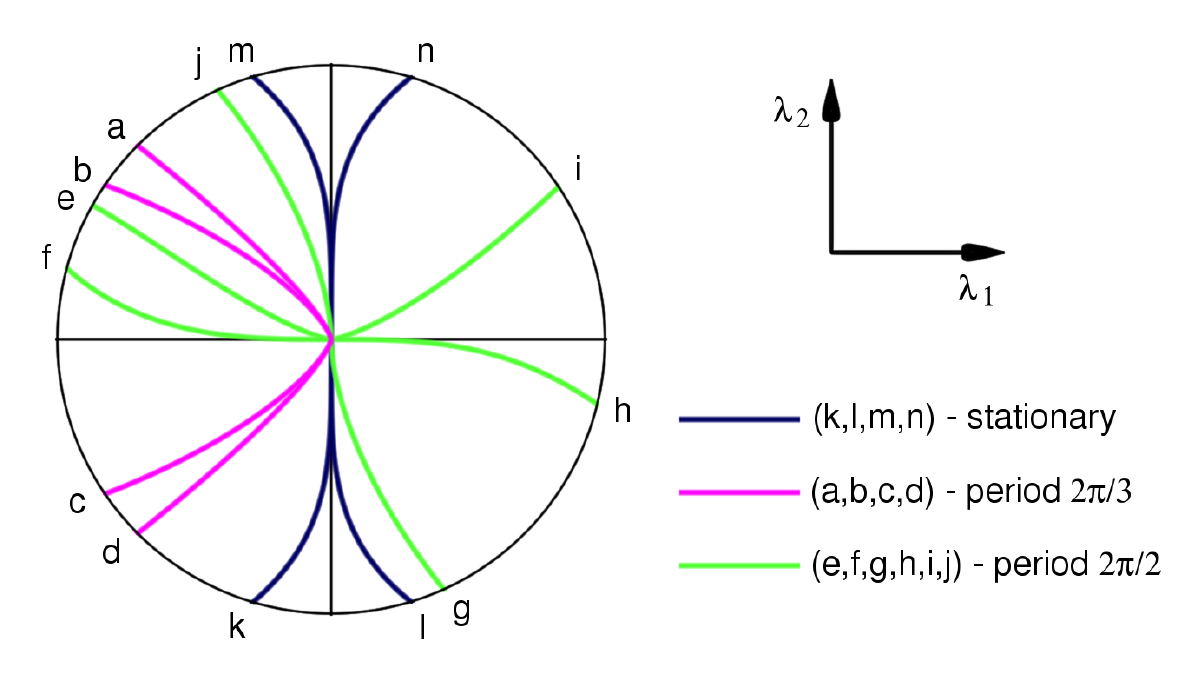}
\caption{The set of those $(\lambda_1,\lambda_2)\in D^2_r,$
$r=0.3,$ for which $(0,(\lambda_1,\lambda_2))\in\R^6\times\R^2$ is
a~global bifurcation point of the system from
Example~\ref{exstat}. The legend on the right describes the minimal
periods of solutions bifurcating from the points of given curve.}
\label{figure2}
\end{figure}

The following results of additional symbolic computations and
estimates ensure that the above conclusions concerning bifurcation
points in $\set{0}\times D^2_r$ hold for every $r\leq 0.3.$
One has $F_j(\lambda)\neq 0$ for every
$j\in\N\backslash\set{0,2,3},$ $\lambda\in U\eqdf (-0.31,0.31)^3.$
The origin is the only singular point of $F_0,$ $F_2,$ $F_3,$ and
$F=F_0\cdot F_2\cdot F_3$ in $U.$ The sets of zeros of $F_0,$
$F_2,$ $F_3$ restricted to $D^2_r\backslash\set{0},$ $r=0.3,$ are
pairwise disjoint and they have no common points with the coordinate
axes. Moreover, the functions $F_0,$ $F_2,$ $F_3,$ and $F$
satisfy the assumptions of Remark~\ref{singtang} for $k=2$ and
$r=0.3.$ Thus for $r=0.3$ every connected component of
$F_0^{-1}(\set{0})\cap D^2_r,$ $F_2^{-1}(\set{0})\cap D^2_r,$
$F_3^{-1}(\set{0})\cap D^2_r,$ and $F^{-1}(\set{0})\cap D^2_r$
contains the origin.

Theorems~\ref{strniezdwh} and~\ref{strniezd} have been also applied
to find bifurcation points in $\set{0}\times D^2_r,$ $r=0.3,$
numerically as zeros of the functions $F_j,$ according to
formulae~\eqref{glbifexsecfir}-\eqref{glbifexsecthir}, which
has been performed by using the program \emph{surf} and presented
on Figure~\ref{figure2}.
The earlier conclusions ensure that the number of curves on
Figure~\ref{figure2}, their localization, and their relative
position do not change when passing to a~smaller scale.

The above results can be summarized as follows. The set of
bifurcation points in $\set{0}\times D^2_r,$ $r=0.3,$ is equal
to the set of global bifurcation points in this domain
and consists of fourteen curves, for which the origin is the only
common point. Apart from the origin four curves (one curve in each
quarter) consist of branching points of nontrivial
stationary solutions (and only such solutions), six curves
(one curve in the first quarter, three curves in the second
quarter, and two curves in the fourth quarter) consist of branching
points of solutions with the minimal period $\pi$
(and only such solutions), and four curves (two curves
in the second quarter and two curves in the third quarter)
consist of branching points of solutions
with the minimal period $\frac{2}{3}\pi$
(and only such solutions). The origin is a~symmetry breaking point,
since it is a~bifurcation point of stationary solutions, solutions
with the minimal periods $\pi,$ and solutions with the minimal
period $\frac{2}{3}\pi$ (as a~cluster point of branching
points of such solutions). The origin is also a~global
bifurcation point of stationary solutions, \mbox{$2$-solu}tions, and
\mbox{$3$-solu}tions. However, it has not been proved that it is
a~branching point of solutions with the minimal periods $\pi$
and $\frac{2}{3}\pi.$
\end{example}

\begin{example}{\label{surfdeg}}
Let the Hamiltonian $H\colon\R^6\times\R^3\rightarrow\R$ be of the
form
\begin{equation*}
H(x,\lambda)\equiv
H(x_1,\ldots,x_6,\lambda_1,\lambda_2,\lambda_3)
=P(x_1,\ldots,x_6,\lambda_1,\lambda_2,\lambda_3)
+Q(x_1,\ldots,x_6),
\end{equation*}
where
\begin{align*}
P(x_1,\ldots,x_6,\lambda_1,\lambda_2,\lambda_3)
&= \frac{1}{2}(7-\lambda_1^4)x_1^2
+\frac{1}{2}(1-\lambda_1^{13})x_3^2
+\frac{7}{2}x_4^2 \\
&\quad +8x_6^2-\lambda_3^4x_1x_3+\lambda_2^3x_4x_6+x_2^4+x_5^4
\end{align*}
and $Q\in C^2(\R^6,\R)$ is the same as in Example~\ref{exdeg},
i.e. it has a~local minimum at the origin and $\nabla^2Q(0)=0,$
see~\eqref{exdegadd} for instance.

$H$ satisfies conditions (H1)-(H3) for $k=3$ and $x_0=0\in\R^6.$

The set of bifurcation points in $\set{0}\times D^3_r$ will be
investigated for $r=0.3.$

One has
\begin{equation*}
A(\lambda_1,\lambda_2,\lambda_3) =\left[
\begin{array}{ccc}
7-\lambda_1^4&0&-\lambda_3^4\\
0&0&0\\
-\lambda_3^4&0&1-\lambda_1^{13}
\end{array}
\right],
\end{equation*}
\begin{equation*}
B(\lambda_1,\lambda_2,\lambda_3) =\left[
\begin{array}{ccc}
7&0&\lambda_2^3\\
0&0&0\\
\lambda_2^3&0&16
\end{array}
\right].
\end{equation*}
Those of the functions $F_j,$ $j\in\N\cup\set{0},$ defined
by~\eqref{Fj}, which vanish at $(0,0,0)\in\R^3$ are $F_0,$ $F_4,$
and $F_7,$ hence $X^+=\set{4,7}.$ One has
\begin{align*}
F_0(\lambda_1,\lambda_2,\lambda_3)&\equiv 0,\\
F_4(\lambda_1,\lambda_2,\lambda_3)
& =-1792\lambda_1^{17}-512\lambda_3^4\lambda_2^3
+8448\lambda_1^{13} -16\lambda_3^8\lambda_2^6+1792\lambda_3^8 \\
&\quad +16\lambda_2^6\lambda_1^{17} -112\lambda_2^6\lambda_1^{13}
-16\lambda_2^6\lambda_1^4+112\lambda_2^6,\\
F_7(\lambda_1,\lambda_2,\lambda_3)
& =-11319\lambda_1^4-5488\lambda_1^{17}-49\lambda_3^8\lambda_2^6
-4802\lambda_3^4\lambda_2^3+5488\lambda_3^8 \\
&\quad +49\lambda_2^6\lambda_1^{17}
-343\lambda_2^6\lambda_1^{13}
-49\lambda_2^6\lambda_1^4+343\lambda_2^6 .
\end{align*}

\begin{figure}
\includegraphics{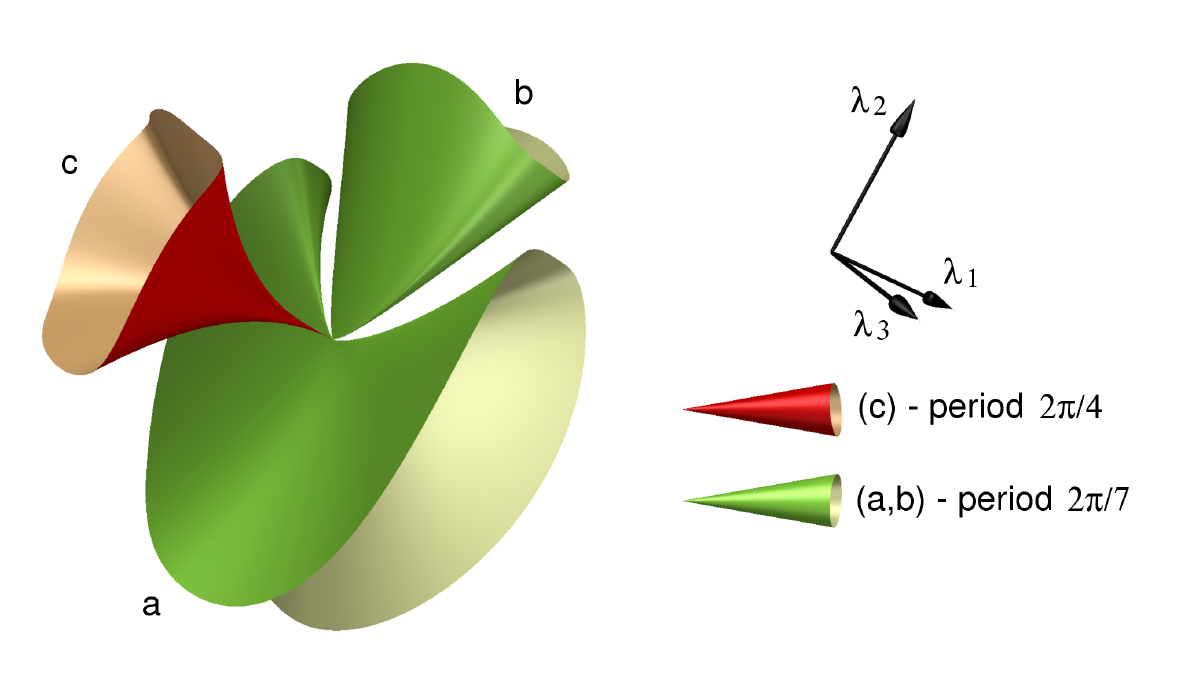}
\caption{The set of those
$(\lambda_1,\lambda_2,\lambda_3)\in D^3_r,$ $r=0.3,$ for which
$(0,(\lambda_1,\lambda_2,\lambda_3))\in\R^6\times\R^3$ is a~global
bifurcation point of the system from Example~\ref{surfdeg}.
The legend on the right describes the minimal periods of solutions
bifurcating from the points of given surface.}
\label{figure3}
\end{figure}

Use will be made of Theorems~\ref{strzdegwh}, \ref{strzdeg}
(see also Remark~\ref{surcorrem}).
Theorems~\ref{strniezdwh}, \ref{strniezd} cannot be applied,
since $F_0=0,$ which means that all the points
$(x_0,\lambda),$ $\lambda\in\R^2,$ are degenerate.

Analogously as in Example~\ref{exdeg} it has been checked that
conditions (E1$(0,\lambda)$), (E2$(0,\lambda)$) are satisfied for
every $\lambda\in U \eqdf (-0.31,0.31)^3.$ Other symbolic
computations and estimates show what follows. One has
$F_j(\lambda)\neq 0$ for every $j\in\N\backslash\set{4,7},$
$\lambda\in U.$ The origin is the only singular point of $F_4,$
$F_7,$ and $F=F_4\cdot F_7$ in $U.$ In particular, the sets
of zeros of $F_4,$ $F_7$ restricted to $D^3_r\backslash\set{0},$
$r=0.3,$ are disjoint. Furthermore, the functions $F_4,$ $F_7,$
and $F$ satisfy the assumptions of Remark~\ref{singtang}
for $k=3$ and $r=0.3.$ Thus for $r=0.3$ every connected component
of $F_4^{-1}(\set{0})\cap D^3_r,$ $F_7^{-1}(\set{0})\cap D^3_r,$
and $F^{-1}(\set{0})\cap D^3_r$ contains the origin.
It has also been checked that $F_4$ and $F_7$ do have zeros
in $D^3_r\backslash\set{0},$ $r=0.3.$

By Theorems~\ref{strzdegwh}, \ref{strzdeg} the following equalities
hold for every $r\leq 0.3.$
\begin{equation}{\label{descbifexfo}}
\begin{aligned}
\Bif(0)\cap D^3_r=\GlBif(0)\cap D^3_r
& =F^{-1}(\set{0})\cap D^3_r \\
& =\left(\GlBif_4^{min}(0)\cup
\GlBif_7^{min}(0)\right)\cap D^3_r, \\
\GlBif_4^{min}\cap D^3_r & =F_4^{-1}(\set{0})\cap D^3_r, \\
\GlBif_7^{min}\cap D^3_r & =F_7^{-1}(\set{0})\cap D^3_r.
\end{aligned}
\end{equation}
The fact that $0\in \GlBif_4^{min}\cap D^3_r$ and
$0\in \GlBif_7^{min}\cap D^3_r$ follows from Lemma~\ref{isotropy}
and Remark~\ref{inters}. (The only minimal periods of nontrivial
solutions in a~neighbourhood of the origin are $\frac{2\pi}{4}$ and
$\frac{2\pi}{7}.$)

The results of numerical application of Theorems~\ref{strzdegwh}
and~\ref{strzdeg}, consisting in finding global bifurcation points
in $\set{0}\times D^3_r,$ $r=0.3,$ as zeros of the functions $F_j,$
according to formulae~\eqref{descbifexfo},
have been obtained by using the program \emph{surf} and presented
on Figure~\ref{figure3}. The earlier conclusions ensure that the
number of the cones on Figure~\ref{figure3} does not change when
passing to a~smaller scale.
\end{example}

\begin{example}{\label{surfstat}}
Let $H\colon\R^6\times\R^3\rightarrow\R$ be the Hamiltonian
defined by
\begin{equation*}
H(x,\lambda)\equiv
H(x_1,\ldots,x_6,\lambda_1,\lambda_2,\lambda_3)
=P(x_1,\ldots,x_6,\lambda_1,\lambda_2,\lambda_3)
+Q(x_1,\ldots,x_6),
\end{equation*}
where
\begin{align*}
& P(x_1,\ldots,x_6,\lambda_1,\lambda_2,\lambda_3) \\
& =\frac{1}{2}(16-85\lambda_1^9+11\lambda_1^5\lambda_3^2
-6\lambda_1^3\lambda_2^2-\lambda_2^3\lambda_3^2
+6\lambda_1^5\lambda_3^4+17\lambda_2^4+\lambda_3^6)x_1^2 \\
&\quad  +\frac{5}{2}x_2^2+x_3^2+\frac{1}{2}x_4^2
+\frac{1}{2}(5+\lambda_1^{13}+8\lambda_2^8)x_5^2 \\
&\quad
+\frac{1}{2}(2\lambda_1^2+4\lambda_3^3\lambda_2^2+\lambda_2^4)x_6^2
+(\lambda_2^3-\lambda_3^2)x_5x_6,
\end{align*}
and $Q$ is defined by the formula~\eqref{exstatadd} from
Example~\ref{exstat}.

$H$ satisfies conditions (H1)-(H3) for $k=3$ and $x_0=0\in\R^6.$

Notice that in this case $x_0=0\in\R^6$ is an isolated critical
point of $H(\cdot,0),$ it is degenerate, and
$\ind{\nabla_xH(\cdot,0)}{0}=0.$

The set of bifurcation points in $\set{0}\times D^3_r$ will be
investigated for $r=0.3.$

Setting
\begin{equation*}
h(\lambda_1,\lambda_2,\lambda_3)\eqdf
16-85\lambda_1^9+11\lambda_1^5\lambda_3^2-6\lambda_1^3\lambda_2^2
-\lambda_2^3\lambda_3^2+6\lambda_1^5\lambda_3^4
+17\lambda_2^4+\lambda_3^6
\end{equation*}
one has
\begin{equation*}
A(\lambda_1,\lambda_2,\lambda_3) =\left[
\begin{array}{ccc}
h(\lambda_1,\lambda_2,\lambda_3) & 0 & 0\\
0 & 5 & 0 \\
0 & 0 & 2
\end{array}
\right],
\end{equation*}
\begin{equation*}
B(\lambda_1,\lambda_2,\lambda_3) =\left[
\begin{array}{ccc}
1 & 0 & 0 \\
0 & 5+\lambda_1^{13}+8\lambda_2^8 & \lambda_2^3-\lambda_3^2 \\
0 & \lambda_2^3-\lambda_3^2
& 2\lambda_1^2+4\lambda_3^3\lambda_2^2+\lambda_ 2^4
\end{array}
\right].
\end{equation*}

\begin{figure}
\includegraphics{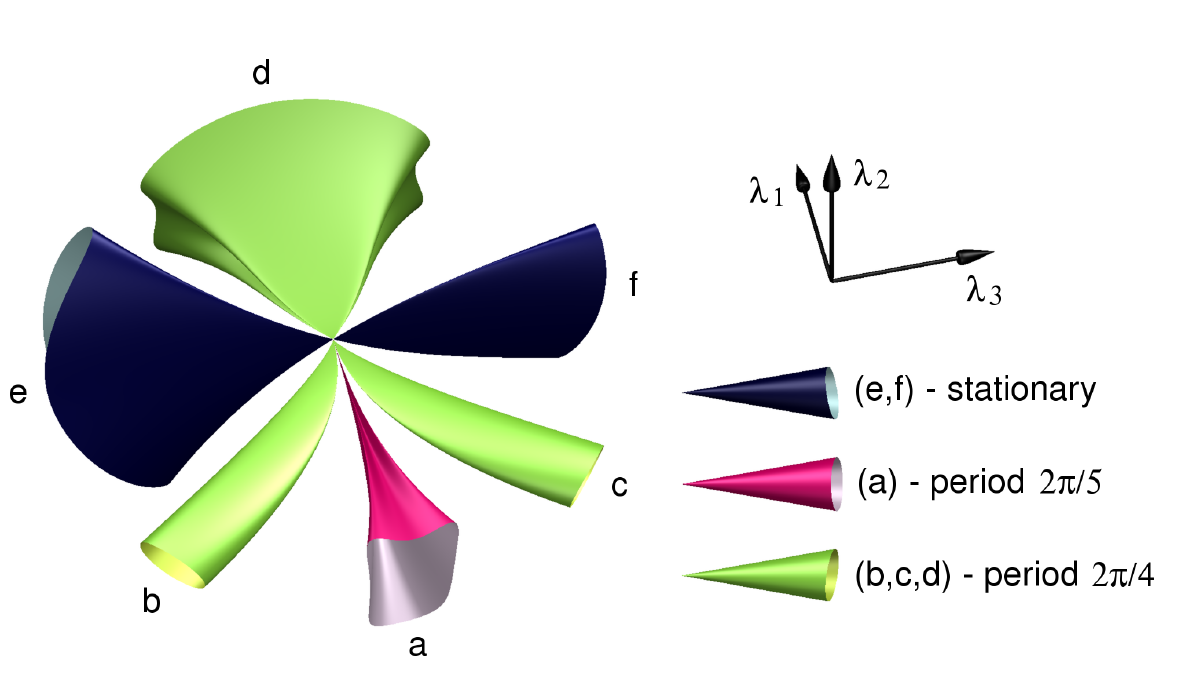}
\caption{The set of those
$(\lambda_1,\lambda_2,\lambda_3)\in D^3_r,$ $r=0.3,$ for which
$(0,(\lambda_1,\lambda_2,\lambda_3))\in\R^6\times\R^3$ is a~global
bifurcation point of the system from Example~\ref{surfstat}.
The legend on the right describes the minimal periods of solutions
bifurcating from the points of given surface.}
\label{figure4}
\end{figure}

Use will be made of Theorems~\ref{strniezdwh}, \ref{strniezd}
(see also Remark~\ref{surcorrem}).
Theorems~\ref{strzdegwh}, \ref{strzdeg} are not suitable
in this case.
(The origin is a~bifurcation point of nontrivial stationary
solutions.)

Those of the functions $F_j,$ $j\in\N\cup\set{0},$ defined
by~\eqref{Fj}, which vanish at $(0,0,0)\in\R^3$ are
$F_0,$ $F_4,$ and $F_5,$ hence $X(0)=\set{0,4,5}.$ One has
\begin{equation*}
F_0=f_0\cdot a_0,\;\;\;F_2=f_2\cdot a_2,\;\;\;F_3=f_3\cdot a_3,
\end{equation*}
where
\begin{align*}
f_0(\lambda_1,\lambda_2,\lambda_3)
& =20\lambda_1^{15}+40\lambda_1^{13}\lambda_3^3\lambda_2^2
+10\lambda_1^{13}\lambda_2^4
+100\lambda_1^2+200\lambda_3^3\lambda_2^2+50\lambda_2^4 \\
&\quad +160\lambda_2^8\lambda_1^2
+320\lambda_2^{10}\lambda_3^3
+80\lambda_2^{12}-10\lambda_2^6
+20\lambda_2^3\lambda_3^2-10\lambda_3^4,\\
f_4(\lambda_1,\lambda_2,\lambda_3)
&=-85\lambda_1^9+11\lambda_1^5\lambda_3^2
-6\lambda_1^3\lambda_2^2-\lambda_2^3\lambda_3^2
+6\lambda_1^5\lambda_3^4+17\lambda_2^4+\lambda_3^6 ,\\
f_5(\lambda_1,\lambda_2,\lambda_3)
&=20\lambda_1^{15}+40\lambda_1^{13}\lambda_3^3\lambda_2^2
+10\lambda_1^{13}\lambda_2^4-125\lambda_1^{13}
+160\lambda_2^8\lambda_1^2 \\
&\quad +320\lambda_2^{10}\lambda_3^3
+80\lambda_2^{12}-1000\lambda_2^8-10\lambda_2^6
+20\lambda_2^3\lambda_3^2-10\lambda_3^4,
\end{align*}
\begin{align*}
a_0(\lambda_1,\lambda_2,\lambda_3)
& =f_4(\lambda_1,\lambda_2,\lambda_3)+16, \\
a_4(\lambda_1,\lambda_2,\lambda_3)
&=f_0(\lambda_1,\lambda_2,\lambda_3)
-64\lambda_1^2-80\lambda_1^{13}-128\lambda_3^3\lambda_2^2 \\
&\quad -32\lambda_2^4-640\lambda_2^8-144, \\
a_5(\lambda_1,\lambda_2,\lambda_3)
&=f_4(\lambda_1,\lambda_2,\lambda_3)-9.
\end{align*}
The functions $a_0,$ $a_4,$ $a_5$ have no zeros in
$U\eqdf (-0.31,0.31)^3.$
Thus $F_0,$ $F_4,$ $F_5$ can be replaced by $f_0,$ $f_4,$ $f_5$ in
computations.

Symbolic computations and estimates show what follows.
One has $F_j(\lambda)\neq 0$ for every
$j\in\N\backslash\set{0,4,5},$ $\lambda\in U\eqdf (-0.31,0.31)^3.$
The origin is the only singular point of $F_0,$ $F_4,$ $F_5,$ and
$F=F_0\cdot F_4\cdot F_5$ in $U.$ In particular, the sets of zeros
of $F_0,$ $F_4,$ $F_5$ restricted to $D^3_r\backslash\set{0},$
$r=0.3,$ are pairwise disjoint. Moreover, the functions
$F_0,$ $F_4,$ $F_5,$ and $F$
satisfy the assumptions of Remark~\ref{singtang} for $k=3$ and
$r=0.3.$ Thus for $r=0.3$
every connected component of $F_0^{-1}(\set{0})\cap D^3_r,$
$F_4^{-1}(\set{0})\cap D^3_r,$
$F_5^{-1}(\set{0})\cap D^3_r,$ and $F^{-1}(\set{0})\cap D^3_r$
contains the origin.
It has also been checked that $F_0,$ $F_4,$ and $F_5$ do have
zeros in $D^3_r\backslash\set{0},$ $r=0.3.$

By Theorems~\ref{strniezdwh}, \ref{strniezd} the following
equalities hold for every $r\leq 0.3.$
\begin{equation}{\label{descbiffouspfir}}
\begin{aligned}
\Bif(0)\cap D^3_r&=\GlBif(0)\cap D^3_r =F^{-1}(\set{0})\cap D^3_r \\
& =\left(\GlBif_0(0)\cup \GlBif_4(0)\cup
\GlBif_5(0)\right)\cap D^3_r,
\end{aligned}
\end{equation}
\begin{equation}{\label{descbiffouspsec}}
\begin{aligned}
\GlBif_0^{min}(0)\cap D^3_r \equiv \GlBif_0(0)\cap D^3_r
& =  F_0^{-1}(\set{0})\cap D^3_r, \\
\GlBif_4(0)\cap D^3_r & =  \left(F_0^{-1}(\set{0})\cup
F_4^{-1}(\set{0})\right)\cap D^3_r, \\
\GlBif_5(0)\cap D^3_r & =  \left( F_0^{-1}(\set{0})\cup
F_5^{-1}(\set{0})\right)\cap D^3_r,
\end{aligned}
\end{equation}
\begin{equation}{\label{descbiffouspthir}}
\begin{aligned}
\GlBif_4^{min}(0)\cap D^3_r\backslash\set{0}
& =  F_4^{-1}(\set{0})\cap D^3_r\backslash\set{0}, \\
\GlBif_5^{min}(0)\cap D^3_r\backslash\set{0}
& =F_5^{-1}(\set{0})\cap D^3_r\backslash\set{0}.
\end{aligned}
\end{equation}

The results of numerical application of Theorems~\ref{strniezdwh}
and~\ref{strniezd}, consisting in finding global bifurcation points
in $\set{0}\times D^3_r,$ $r=0.3,$ as zeros of the functions $F_j,$
according to
formulae~\eqref{descbiffouspfir}-\eqref{descbiffouspthir},
have been obtained by using the program \emph{surf} and presented
on Figure~\ref{figure4}.
The earlier conclusions ensure that the number of the cones on
Figure~\ref{figure4} does not change when passing to a~smaller
scale.
\end{example}

\begin{remark}{\label{surcorrem}}
Corollaries~\ref{multzdeg}, \ref{multstat}, and results
from~\cite{SfMA,LS} can be used in Examples~\ref{surfdeg}
and~\ref{surfstat} to verify the number of the cones forming
the set of bifurcation points and to confirm that the origin is
a~symmetry breaking point.
\end{remark}

Examples analogous to Examples~\ref{exdeg},~\ref{exstat} can be
constructed for any number of degrees of freedom,
whereas examples similar to
Examples~\ref{surfdeg},~\ref{surfstat} can be given for
any number of parameters and any number of degrees of freedom.

\appendix

\section{Description of semianalytic sets}{\label{semianal}}

In this appendix the relevant results from~\cite{Sf},
in the case they have been used in
Sections~\ref{structsection},~\ref{examples}, are
summarized for the convenience of the reader.

In what follows use is made of Definition~\ref{admissible} of
admissible function and Definition~\ref{testfuncdef} of test
function.

As well known, if $F\colon\R^2\rightarrow\R$ is an admissible
mapping then for sufficiently small $r>0$ the set
$F^{-1}(\set{0})\cap D^2_r\backslash\set{0}$
is either empty or it is a~union of finitely
many disjoint analytic curves, each of which meets the origin and
crosses $S^1_r$ transversally in one point.

If $g\colon\R^2\rightarrow\R$ is a~test function for an admissible
mapping $F\colon\R^2\rightarrow\R$ then for sufficiently small
$r>0$ the function $g$ has a~constant sign on every connected
component of the set $F^{-1}(\set{0})\cap D^2_r\backslash\set{0}$
(i.e. on each of the analytic curves forming this set).

The following notation is used.
\begin{trivlist}
\item $b(F)=$ the number of components of the set
      $F^{-1}(\set{0})\cap D^2_r\backslash\set{0},$
\item $b_{+}(g,F)=$ the number of components of
      $F^{-1}(\set{0})\cap D^2_r\backslash\set{0}$ on which $g$ is
      positive,
\item $b_{-}(g,F)=$ the number of components of
      $F^{-1}(\set{0})\cap D^2_r\backslash\set{0}$ on which $g$ is
      negative.
\end{trivlist}

Clearly, $b_{+}(g,F)+b_{-}(g,F)=b(F).$

Let $\Jac(g,F) \colon \R^2\rightarrow\R$ be the Jacobian
of the mapping $(g,F)\colon\R^2\rightarrow\R^2,$ and let the mapping
$h(g,F)\colon\R^2\rightarrow\R^2$ be defined by
\begin{equation*}
h(g,F)=(\Jac(g,F),F).
\end{equation*}

In the following theorem $\ind{h(g,F)}{0}$ denotes the topological
index of $0\in\R^2$ with respect to $h(g,F)$
(see Section~\ref{degree}).

\begin{theorem}[\cite{Sf}]{\label{szafr}} If
$g\colon\R^2\rightarrow\R$ is a~test function for an admissible
mapping $F\colon\R^2\rightarrow\R$ then $0\in\R^2$ is isolated in
$h(g,F)^{-1}(\set{0})$ and
\begin{equation}{\label{szafrform}}
b_{+}(g,F)-b_{-}(g,F)=2\cdot\ind{h(g,F)}{0}.
\end{equation}
\end{theorem}

\begin{corollary}[\cite{Sf}]{\label{wnszafr}}
If $g_+\colon\R^2\rightarrow\R$ is a~nonnegative test function
for an admissible mapping $F\colon\R^2\rightarrow\R$ then
$0\in\R^2$ is isolated in $h(g_+,F)^{-1}(\set{0})$ and
\begin{equation}{\label{wnszafrform}}
b(F)=b_+(g_+,F)=2\cdot\ind{h(g_+,F)}{0}.
\end{equation}
\end{corollary}

Let $b_i(F),$ $g_i,$ $i=1,\ldots,4,$ be such as in
Remark~\ref{probnerem}.

\begin{corollary}[\cite{IR}]{\label{wnIRcw}}
If an admissible mapping $F\colon\R^2\to\R$ has no zeros
on the coordinate axes in a~neighbourhood of the origin then
\begin{align*}
b_1(F)+b_2(F)+b_3(F)+b_4(F) & =  2\cdot\ind{h(g_1,F)}{0}, \\
b_1(F)-b_2(F)-b_3(F)+b_4(F) & =  2\cdot\ind{h(g_2,F)}{0}, \\
b_1(F)+b_2(F)-b_3(F)-b_4(F) & =  2\cdot\ind{h(g_3,F)}{0}, \\
b_1(F)-b_2(F)+b_3(F)-b_4(F) & =  2\cdot\ind{h(g_4,F)}{0}. \\
\end{align*}
\end{corollary}

\section*{Acknowledgements}

This paper is based on a~part of the author's PhD
thesis~\cite{RPhD}. The author wishes to express his gratitude
to his thesis advisor, Professor S{\l}awomir Rybicki, for his
remarks.

\end{document}